\documentclass[10pt]{article}

\usepackage{url}
\usepackage{mathtools}

\usepackage{amssymb}
\usepackage{amsthm}
\usepackage{empheq}
\usepackage{latexsym}
\usepackage{enumitem}
\usepackage{eurosym}
\usepackage{dsfont}
\usepackage{appendix}
\usepackage{color} 
\usepackage[unicode]{hyperref}
\usepackage{frcursive}
\usepackage[utf8]{inputenc}
\usepackage[T1]{fontenc}
\usepackage{geometry}
\usepackage{multirow}
\usepackage{todonotes}
\usepackage{lmodern}
\usepackage{anyfontsize}
\usepackage{pgfplots}
\usepackage{stmaryrd}
\usepackage{cleveref}
\usepackage{natbib}
\usepackage{upgreek}
\usepackage{comment}
\usepackage{subfig}
\usepackage{scalerel}

\RequirePackage[subfigure]{tocloft}
\advance\cftbeforesecskip 0.5em\relax
\usepackage[nottoc]{tocbibind}

\input Acorn.fd

\pgfplotsset{compat=1.16}

\setcitestyle{numbers,open={[},close={]}}

\definecolor{red}{rgb}{0.7,0.15,0.15}
\definecolor{green}{rgb}{0,0.5,0}
\definecolor{blue}{rgb}{0,0,0.7}
\hypersetup{colorlinks, linkcolor={red},citecolor={green}, urlcolor={blue}}
			
\makeatletter \@addtoreset{equation}{section}

\newtheorem{theorem}{Theorem}[section]
\newtheorem{assumption}[theorem]{Assumption}
\newtheorem{corollary}[theorem]{Corollary}

\newtheorem{lemma}[theorem]{Lemma}
\newtheorem{proposition}[theorem]{Proposition}

\newtheorem{definition}[theorem]{Definition}
\newtheorem{remark}[theorem]{Remark}

\DeclareUnicodeCharacter{014D}{\=o}
\def \E{\mathbb{E}}
\def \F{\mathbb{F}}

\def \H{\mathbb{H}}

\def \L{\mathbb{L}}

\def \N{\mathbb{N}}
\def \P{\mathbb{P}}
\def \Q{\mathbb{Q}}
\def \R{\mathbb{R}}

\def\Ac{{\cal A}}
\def\Bc{{\cal B}}
\def\Cc{{\cal C}}

\def\Fc{{\cal F}}

\def\Hc{{\cal H}}
\def\Ic{{\cal I}}

\def\Mc{{\cal M}}

\def\Oc{{\cal O}}
\def\Pc{{\cal P}}

\def\Tc{{\cal T}}

\def\Vc{{\cal V}}
\def\Wc{{\cal W}}

\def\Yc{{\cal Y}}
\def\Zc{{\cal Z}}

\def\eps{\varepsilon}
\def\d{{\mathrm{d}}}

\newcommand{\1}[1]{\mathds{1}_{\{#1\}}}
\def\e{{\mathrm{e}}}

\def\=c{ \overset{c}{=}}
\DeclareRobustCommand{\varlambda}{\text{\usefont{OML}{txmi}{m}{it}\symbol{"15}}}

\def\as{\text{\rm--a.s.}}
\def\ae{\text{\rm--a.e.}}

\DeclareMathOperator*{\argmin}{arg\,min}

\newcommand{\xdim}{m}

\newcommand{\muin}{\mu_{0}}
\newcommand{\mufin}{\mu_{T}}

\renewcommand{\inf}{\mathop{\mathrm{inf}\vphantom{\mathrm{sup}}}}

\allowdisplaybreaks

\begin{document}

\title{Marginal flows of non-entropic weak Schr\"odinger bridges\footnote{\today.}}


\author{Camilo {\sc Hern\'andez} \footnote{University of Southern California, ISE department, USA. camilohe@usc.edu.} \and Ludovic {\sc Tangpi} \footnote{Princeton University, ORFE department, USA. ludovic.tangpi@princeton.edu.}}

\date{\vspace{-5ex}}

\maketitle

\abstract{
{
This paper introduces a dynamic formulation of divergence-regularized optimal transport with weak targets on the path space.
In our formulation, the classical relative entropy penalty is replaced by a general convex divergence, and terminal constraints are imposed in a weak sense.
We establish well-posedness and a convex dual formulation, together with a dual existence result and explicit structural characterizations of primal and dual optimizers. 
Specifically, the optimal path measure admits an explicit density relative to a reference diffusion, generalizing the classical Schrödinger system. 
In the case of zero transport cost, which corresponds to a non-entropic dynamic Schrödinger problem, we further characterize the flow of time marginals of the optimal bridge, recovering known results in the entropic setting and providing new descriptions for non-entropic divergences, including the $\chi^2$--divergence.}
}

\section{Introduction}\label{sec:intro}

Optimal transport (OT) provides a powerful framework for comparing probability measures, and has become a cornerstone in modern statistics, data science, machine learning and operations research. 
Formulated as a convex optimization problem over couplings between two distributions $\muin,\mufin$ not only compares distributions, but also exhibits geometric properties of the set of probability measures, and plays a fundamental role in partial differential equations.
However, a well-known limitation is its poor statistical scalability: 
When the distributions $\muin, \mufin$ are replaced by empirical measures based on i.i.d. samples, the convergence of the empirical OT problem to its population counterpart deteriorates exponentially with the dimension, an instance of the curse of dimensionality; see, e.g., \citet{chewi2024statistical} for a recent survey. 
To overcome this limitation, entropic regularized optimal transport introduces an entropic penalty to the OT objective, yielding smoother dual problems and faster statistical convergence, see for instance \cite{bernton2021entropic,bernton2022stability,cuturi2013sinkhorn,peyre2019computational,mena2019statistical} among many works on the subject.
The entropic regularized perspective exhibit great statistical properties. 
In fact, the empirical optimal costs, optimal couplings, and dual potentials all converge to their population counterparts at the rate $\Oc(\sqrt{N})$ when the measures $\muin,\mufin$ are approximated by empirical measures of $N$ i.i.d. random variables, see \citet{genevay2019sample} and \citet{mena2019statistical}.

\medskip

Despite these advantages, the entropy penalty unfortunately produces full-support couplings (a consequence of Brenier's theorem), a phenomenon known as overspreading, which 
can lead to undesirable blurring effects in imaging and manifold learning applications \citet{blondel2018smooth}.
Moreover, entropic OT becomes numerically unstable when the regularization parameter is small, as the dual variables may attain exponentially large or small values \citet{li2020continuous}.
To overcome these shortcomings, an increasingly popular alternative to entropic regularization is quadratically regularized optimal transport, which corresponds to the problem
\begin{equation}
\label{eq:static.ROT}
	v(\muin,\mufin) := \inf_{\pi \in \Pi(\muin,\mufin)}\bigg( \int_{\R^\xdim\times\R^\xdim} c(x,y)\pi(dx,dy) + \varepsilon \int_{\R^\xdim\times \R^\xdim}\ell\Big(\frac{d\pi}{d\muin\otimes \mufin} \Big )\muin\otimes \mufin(dx,dy) \bigg)
\end{equation}
where $\Pi(\muin,\mufin)$ is the set of couplings of the probability measures $\muin,\mufin$ and  $\ell(x):=\frac12|x|^2$.
This problem has attracted a sustained interest in recent years. 
While theoretical aspects of this problem are a lot less well-understood than its entropic counterpart, it has already been showed notably that quadratic OT retains many computational benefits of entropic OT and produces sparse approximations that more faithfully reflect the geometry of unregularized OT, see \citet{gonzalez2024sparsity,wiesel2025sparsity}. 
We refer for instance to \cite{nutz2025quadratically,gonzalez2024sparsity,zhang2023manifold,wang2025quadratic,wiesel2025sparsity,essid2018quadratically} for recent works on sparsity, and \citet{gonzalez2025linear,gonzalez2025sample} on gradient descent and sample approximation.
Quadratic regularized OT also enjoys numerical stability even for small regularization parameters and admits fast gradient-based optimization as showed by \citet{gonzalez2025linear}.
While quadratically regularized OT may suffer the curse of dimensionallity, the interesting recent paper of \citet{gonzalez2025sparse} shows that for a sufficiently large class of convex functions $\ell$ the divergence-regularized OT \eqref{eq:static.ROT} overcomes it and the empirical version of the problem converges at the parametric rate.
This further motivates the study of \eqref{eq:static.ROT} beyond the quadratic case.
Divergence regularized OT is also an emerging research topic, with recent references including \cite{bayraktar2025stability,gonzalez2025sparse,seguy2017large,muzellec2017tsallis,dessein2018regularized}.
Research on divergence-regularized OT has been restricted to the static case.
The goal of this work is to introduce divergence-regularized OT in the dynamic setting, and to derive the structure of the optimizer.

\paragraph*{Divergence-regularized OT and weak constraints.}
Let us introduce the divergence on the path space.
Let $\xdim \in \N$, $T>0$ and $\muin,\mufin $ be given Borel probability measures on $\R^m$. 
Let $(\Omega,\F,\P)$ be a filtered probability space suppoting a Brownian motion $W$ and such that $\P$ is the unique weak solution of the stochastic differential equation 
\[
\d X_t = b(t,X_{\cdot\wedge t}) \d t + \d W_t, \, \P\circ X_0^{-1}=\muin.
\]
The process $X$ belongs to $\Cc_T$, the space of continuous functions on $[0,T]$ with values in $\R^\xdim$ and $X_{\cdot\wedge t}$ denotes the path of $X$ up to time $t\in [0,T]$.
The (dynamic) divergence-regularized OT is 
\begin{equation}
\label{eq:Reg.OT}
	\Vc^\varepsilon(\muin,\mufin) = \inf\Big\{\E^\Q\big[C(X)\big] + \varepsilon   \Ic_\ell(\Q|\P),\, \Q\circ X_0^{-1}=\muin, \,\Q\circ X_T^{-1}=\mufin \Big\},
\end{equation}
where $C:\Cc_T\to \R$ is a given cost function and $\Ic_\ell(\cdot |\P)$ denotes the divergence operator with respect to the measure $\P$ and the function $\ell$ i.e.
\begin{equation}
\label{eq:def.divergence.intro}
\Ic_\ell(\Q| \P):=\begin{cases} \displaystyle\E\bigg[\ell\Big( \frac{\d \Q}{\d \P}\Big)\bigg] & \text{if }\Q\ll\P\\
\infty&  \text{otherwise}.
\end{cases}
\end{equation}
In essence,
we consider the divergence-regularization of the celebrated Benamou–Brenier continuous-time transport problem \cite{benamou2000computational}.
This viewpoint 
reveals deep connections between optimal transport, fluid mechanics, and gradient flows, but also enables powerful analytical and numerical tools that are unavailable in the static case.
Obviously when $C=0$, this problem corresponds to a form of (non-entropic) Schr\"odinger problem \cite{follmer1988random,schrodinger1931uber,schrodinger1932sur}  in which the discrepancy between probability measures is considered with respect to a general divergence.
This is in its own right an interesting generalization of Schr\"odinger problem that, as far as we know, has only been considered {in a static setting by \citet{leonard2001minimizers} and \citet{backhoff2022nonlinear}}.
From a probabilistic standpoint, divergences arise naturally when one seeks to control deviations between laws without imposing symmetry or additivity properties inherent to relative entropy. 
Their flexibility makes them well-suited for capturing different aspects of distributional distance, while still retaining structural features such as convexity and monotonicity under Markov kernels.
However, they do not, in general, satisfy the so-called data-processing equality
\begin{equation}
\label{eq:data.processing}
\Ic_\ell(\Q|\P) = \Ic_\ell (\Q_{0T}|\P_{0T}) + \int_{(\R^\xdim)^2}\Ic_\ell(\Q_{xy}|\P_{xy}) \P_{0T}(\d x\d y),
\end{equation}
where $\Q_{xy}(\cdot) \coloneqq \Q(\cdot|X_0 =x, X_T = y)$ and $\Q_{0T} \coloneqq \Q\circ (X_0,X_T)^{-1}$; which is well-known to hold (essentially) only in the case $\ell(x) = x\log(x)-x + 1$, which corresponds to the relative entropy see, e.g., \citet*[Theorem 2]{leonard2014some} and {\rm \citet*[Theorem 5.5 and Remark 5.6]{lacker2018law}}.
Consequently, the dynamic problem \eqref{eq:Reg.OT} does not necessarilty have the same value as its static counterpart \eqref{eq:static.ROT} for arbitrary $\ell$ (as is known in the entropic case), see for instance the arguments of proof of \cite[Proposition 2.3]{leonard2014survey}.

\medskip

A key feature of \eqref{eq:Reg.OT} is that the divergence regularization allows mass to diffuse along many possible paths rather than following a single deterministic trajectory, hence the reward accounts for the average trajectories induced by a path distribution $\Q$, not the precise microscopic path itself.
	This naturally suggests considering problems where we no longer require $X$ to hit the target distribution exactly, i.e., $\Q\circ X_T^{-1} = \mufin$, but to match it on average.
\emph{Weak optimal transport} problems introduced by Gozlan et. al. \cite{gozlan2017kantorovich,gozlan2018characterization} have emerged as a particularly flexible, yet tractable framework for modeling broader relationships between distributions.
See \eqref{eq.weaktransport} for precise definition and \citet{backhoff2022applications} for references on the topic and a sample of applications. 
Specifically, given a weak cost $c:\R^\xdim \times \Pc_p(\R^\xdim)\to \R$ and associated weak optimal transport value $\Wc_c(\cdot,\cdot)$, we will focus on the more general problem
\begin{equation}
\label{eq:Reg.OT.weak}
	V_c(\muin,\mufin) = \inf\Big\{\E^\Q\big[C(X)\big] +    \Ic_\ell(\Q|\P),\, \Q\circ X_0^{-1}=\muin, \,\Wc_c(\Q\circ X_T^{-1},\mufin) = 0 \Big\}
\end{equation}
where, for simplicity, we let $\varepsilon = 1$.
Examples will be discussed in subsection \ref{sec:examples} where we will see that adequate choices of $c$ allow to recover the constraint $\Q\circ X_T^{-1}=\mufin$, but also to model cases where the terminal law $\Q\circ X_T^{-1}$ of $X$ is constrained to be in convex order with $\mufin$ or to belong to a ball around it.

\paragraph*{Main results.}
In this work, we derive the dual problem of \eqref{eq:Reg.OT.weak} in the sense of convex analysis and, as main contributions of the paper, we recover fundamental structural properties of the primal and dual optimizers of \eqref{eq:Reg.OT.weak}.
	Let us also mention that we establish conditions for the existence of dual optimizers.
Let us give a brief synopsis of the main results of the work here; precise statements and assumptions are given in Section \ref{sub:main_results}.
For a generic probability measure $P$ and $\ell^*$ being the convex conjugate of $\ell$, let us introduce the functions
\begin{equation*}
	\Phi_{P}(\xi) := \inf_{r \in \R}(\E^{P}[\ell^\ast(\xi - r)] + r),\mbox{ and, } Q_c\varphi(x) := \inf_{\rho \in \Pc_p(\R^\xdim)}(c(x, \rho) + \langle \varphi,\rho\rangle)\quad x \in \R^\xdim,\; \xi \in \L^0.
\end{equation*}
Let $\F = (\Fc_t)_{t\in [0,T]}$ be the canonical filtration.
\Cref{thm:exists.duality} shows the duality relationship
\begin{equation*}
	V_c(\muin,\mufin) = \sup_{\varphi\in C_{b,p}(\R^m)}\bigg(  -\int_{\R^\xdim } \Phi_{\P_x}\big( - Q_c\varphi(X_T)- C(X) \big)\muin(dx) - \int_{\R^\xdim} \varphi(x)\mufin(dx) \bigg) + \Ic_\ell(\muin|\nu_0)
\end{equation*}
where $(\P_x)_{x\in \R^\xdim}$ is the regular conditional distribution of $\P$ with respect to $\Fc_0$, and $C_{b,p}(\R^\xdim)$ is a set of continuous functions defined below.
\Cref{thm.existence.dual} establishes conditions on the data of the problem, i.e., $\ell,C$ and $c$, for the existence of a dual optimizer $\varphi$ in the above characterization.
Then, \Cref{thm:propertyMain} shows that if $\varphi$ is a dual optimizer, then the probability measure $\Q$ given by
\begin{equation}
\label{eq:charac.optim.intro}
	\dfrac{\d\Q\!}{\d \P} :=\frac{\d\muin }{\d\nu_0} (X_0)   \partial_x\ell^*\big(  - Q_c\varphi(X_T) - C(X) -\psi(X_0)\big),
\end{equation}
for some measurable function $\psi$, is a primal optimizer.
Conversely, when $c(x,\rho)=\int_{\R^\xdim} \1{x\neq y}\rho(\d y)$, \Cref{thm:property.main.2} shows that if an admissible measure $\Q$ takes the form \eqref{eq:charac.optim.intro} then it is primal optimal and the function $\varphi$ is dual optimal.
Moreover, the functions $\varphi$ and $\psi$ satisfy an associated Schr\"odinger system, see \Cref{eq.prop.system}.
Lastly, in the case of Schr\"odinger problem with divergence cost, i.e. when $C=0$, we can further characterize the time-marginals of the optimal bridge.
In fact, see \Cref{them:charac.marginal}, in general the probability density function $\Q_t$ of $\Q\circ X_t^{-1}$ takes the form
\begin{equation}
\label{eq:non-markovian.charac.statement.intro}
	\partial_x \log(\Q_t(x)) = \E^{\Q}\Big[ \alpha^\star_t  + \overleftarrow{\alpha}_{T-t} \circ \overleftarrow{\Tc}^{-1}  |X_t = x \Big] + b(t,x)+ b(T-t,x),  \; \d t\otimes \d \Q_t\ae
\end{equation}
for two predictable processes $\alpha^\star$ and $\overleftarrow{\alpha}$, where $\overleftarrow{\Tc}$ is the time reversal mapping on the path space. 
The explicit form of this formula can be obtained on a case-by-case basis, leveraging stochastic optimal control theory.
This is illustrated in \Cref{cor:quad.cha} for the case of the $\chi^2$-divergence, see \eqref{flow.chisquare}.

\paragraph*{Related literature.}
In the entropic case $\ell(x) = x\log(x)-x + 1$, and with $C=0$, our results build on the well-established literature on the dynamic Schrödinger's problem. 
	In particular, the characterization of the density of the optimal bridge in terms of potentials, the Schrödinger system, as well as the characterization of the flow of marginal laws (with respect to a reversible reference measure) in terms of Hamilton--Jacobi--Bellman equations are well understood in this setting, see the survey of \citet{leonard2014survey} and the references therein.
	In the non-entropic case, a first step towards understanding the infinitesimal behaviour of the quadratically penalized optimal transport was discussed in \citet{garriz2024infinitesimal}
	\medskip

In the static setting, Schrödinger problems for general divergences were studied in \citet{leonard2001minimizers}. 
Regarding the study of divergence-penalized optimal transport problems, as mentioned above, an emerging body of literature investigates the effects of general penalization on the transport map.
In particular, what can be construed as the static versions of our primal and dual existence results as well as the characterization of the density of the optimal static bridge, appear in the work of \citet{bayraktar2025stability} for bounded cost function, in \cite{gonzalez2025sparse} for compactly supported $\muin$ and $\mufin$, and under fairly general assumptions in \citet{nutz2025quadratically} for quadratically penalized optimal transport, that is for $\chi^2$-divergence.

\medskip
We remind the reader that the previous references are all in the static case, whereas the setting of this paper corresponds to the dynamic non-entropic case, which, to the best of our knowledge, has remained unexplored.
Moreover, to the best of our knowledge, the present work seems to be the first Schr\"odinger problems with weak targets.

\paragraph*{Organization of the paper.}
In the next section, we present the divergence-regularized optimal transport we are interested in, with weak terminal constraints, and state our main results before providing a few examples.
\Cref{sec:preliminaries_and_examples} gathers some preparatory results for the proofs of our main results, which are presented in Section \ref{sec:proofs_of_the_main_results}.
Some standard but frequently used results from convex analysis and variational calculus are gathered in the appendix.

\paragraph*{Frequently used notation.}
Throughout this work, $\Omega\coloneqq \Cc([0,T],\R^\xdim)$ denotes the canonical space of continuous paths on $[0,T]$ with values in $\R^m$.
	$X$ denotes the canonical process, i.e., $X_t(\omega) \coloneqq \omega(t),(t,\omega)\in [0,T]\times\Omega$, and canonical filtration $\F \coloneqq (\Fc_t)_{t\in [0,T]}$, $\Fc_t\coloneqq \sigma(X_s:s\in [0,t])$. 
	$\L^0$ denotes the collection of real-valued random variables $\xi$ on $(\Omega,\Fc_T)$.
	Given $\Q\in {\rm Prob}(\Omega)$, we denote by $(\Q_x)_{x\in \R^m}$ the regular conditional probability distribution of $\Q$ given $\Fc_0$.\medskip

For a function $\ell:\R \longrightarrow \R$, its convex conjugate is given by $\ell^\ast(x)\coloneqq \sup_{y\in \R} (xy-\ell(y))$.
We also introduce the sets
\[
B_{b,p}(\R^\xdim):=\{\varphi:\R^\xdim\longrightarrow \R \text{ Borel measurable s.t. }   \exists a, b, \in \R: a\leq \varphi(y)  \leq b(1+|y|^p) ,\, \forall y\in \R^\xdim\}.
\]
and $C_{b,p}(\R^\xdim):=\{\varphi \in B_{b,p}(\R^\xdim) \text{ continuous}\}$.\medskip

Given a Polish space $E$, we denote by $\Pc(E)$ the set of Borel probability measures on $E$ and $\Pc_p(E)$ the set of elements of $\Pc(E)$ with finite $p^{\mathrm{th}}$ moments. 
	For a $\mu$-integrable function $\varphi$ defined on $E$, we put $\langle \varphi, \mu\rangle:=\int_E\varphi(x)\mu(\d x)$.
	For $\mu,\nu\in \Pc(E)$, $\Pi(\mu,\nu)$ is the set of couplings between $\mu$ and $\nu$ and
	the $p$-Wasserstein distance given by
\(
W_p^p (\mu,\nu)\coloneqq \inf_{\pi\in \Pi(\mu,\nu)}\int_{E \times E} \|x-y\|^p  \pi(\d x, \d y).
\)

\section{Problem statement}

Let $T>0$, $m$ a nonnegative integer, and $\nu_0 \in \Pc(\R^m)$ be fixed. 
We assume to be given a mapping $b:[0,T]\times \Cc([0,T],\R^\xdim) \longrightarrow \R^m$, such that there is a unique $\P\in {\rm Prob}(\Omega)$ weak solution to the SDE
\begin{equation}\label{eq.dyn.P}
	\d X_t = b(t,X_{\cdot\wedge t}) \d t +  \d W_t, \, \P\circ X_0^{-1}=\nu_0,
\end{equation}
where $W$ denotes a $\P$--Browninan motion.
	This is the case, for instance, if either $b$ is locally Lipschitz or bounded, see \cite[Section III.2d]{jacod2003limit}.\medskip

We are also given coefficients $C$, $\ell$, and $c$, satisfying the following set of assumptions.

\begin{assumption}\label{assump.data}
\begin{enumerate}[label=$(\roman*)$, ref=.$(\roman*)$,wide,  labelindent=0pt]
\item \label{assump.data.1}The function $\ell:\R\longrightarrow \R_+$ is strictly convex and twice continuously differentiable on $(0,\infty)$, satisfying $\ell(1)=0$ and $\ell(x)/x\to \infty$ as $x\to \infty$. 
Moreover, its convex conjugate $\ell^\ast$ satisfies 
	\[
	\forall r\in \R, \exists \gamma>0, \exists x_0\in \R: \forall x\geq x_0, \; \ell^\ast(x+r)\leq \gamma \ell^\ast(x).
	\]

\item \label{assump.data.2} There is $p\geq1$ such that $c:\R^\xdim\times \Pc_p(\R^\xdim)\longrightarrow \R_+$ is jointly l.s.c.~with respect to the product topology of $\R^\xdim\times\Pc_p(\R^\xdim)$, $\Pc_p(\R^\xdim)$ being equipped with the topology generated by $W_p$. There is $L>0$ such that
\[
c(x,\rho)\leq L\bigg(1+\|x\|^p +\int_{\R^\xdim}\|y\|^p\rho(\d y)\bigg), \, (x,\rho)\in \R^\xdim\times\Pc_p(\R^\xdim).
\]
In addition, the map $\rho\longmapsto c(x,\rho)$ is linearly convex for any $x\in \R^\xdim$, i.e.,
\[
c(x,\lambda\rho_1+(1-\lambda)\rho_2)\leq \lambda c(x,\rho_1)+(1-\lambda )c(x,\rho_2),\, \mbox{ for all }  \rho_1,\rho_2\in \Pc_p(\R^\xdim),\, \lambda\in [0,1].
\]
\item The function $C$ is Borel-measurable and bounded from below.
\end{enumerate}
\end{assumption}

Observe that under the above assumptions, the convex conjugate $\ell^\ast$ is continuously differentiable. 

\paragraph*{Weak targets.}
In this work, we will consider transportation problems for which the terminal configuration is specified in a broader sense by leveraging the weak optimal transport problem (WOT) first considered by \citet*{gozlan2017kantorovich}.
	Introduced as an equivalent tool to derive novel concentration inequalities, this generalization of optimal transport enables the construction of couplings between probability measures that possess desired structures or properties.
	Given $c:\R^\xdim\times\Pc(\R^\xdim)\longrightarrow \R_+$, as above, WOT is given by
\begin{align}\label{eq.weaktransport}
 \Wc_c (\mu,\nu):=\inf_{\pi\in \Pi(\mu,\nu)}\int_{\R^\xdim} c(x,\pi_x)\mu(\d x).
\end{align}
We write $\Wc_c$ to stress the dependence of WOT in $c$.
	Notice that when $c(x,\rho)=\int_{\R^\xdim} \|x-y\|^p\rho(\d y)$, for $p\geq 1$, $\Wc_c^{1/p}$ is nothing but the $p$-Wasserstein distance \(W_p \).
	Further details, results relevant to our analysis, and examples {\color{black}are deferred to \Cref{sec:examples}. 
	Recalling that $\mu=\nu\in \Pc_p(\R^\xdim) \Longleftrightarrow  W_p(\mu,\nu)=0$, it seems natural to introduce a binary relation induced by $\Wc_c$ over the elements of $\Pc(\R^\xdim)$.

\begin{definition}\label{def.msim}
Given any two probability measures $\mu,\nu\in\Pc(\R^\xdim)$, we write $\mu \=c \nu $ whenever $  \Wc_c(\mu,\nu)=0$. 
\end{definition}

\paragraph*{The non-entropic optimal transportation problem.}


Given two probability measures $\muin, \mufin$ such that $\Ic(\muin|\nu_0)<\infty$,
the goal of this work is to analyze the optimal transport problem
\begin{equation}\label{eq.SP}
	V_c(\muin,\mufin) \coloneqq \inf_{\Q \in \Pc_c(\muin,\mufin)}\Big( \E^{\Q}[C(X)]+\Ic_\ell(\Q|\P)\Big),
\end{equation}
with $\Ic_\ell$ denoting the divergence operator \eqref{eq:def.divergence.intro}
and 
\begin{equation*}
		\Pc_c(\muin,\mufin) \coloneqq  \Big\{ \Q \in \Pc(\muin) :  \Q \circ X_T^{-1} \overset{c}{=} \mufin \Big\}, \text{ where, } \Pc(\muin) \coloneqq  \big\{ \Q \in {\rm Prob}(\Omega):  \Q\circ X_0^{-1}=\muin  \big\}.
\end{equation*}

\section{Main results and Examples}
\label{sub:main_results}

We will now rigorously state the main results of this work.
In addition to deriving well-posedness and convex duality for the divergence-regularized transport problem, we are chiefly interested in explicit representations of the optimal transport plan in terms of the optimal potential.
This property is a crucial building block for deriving the Sinkhorn algorithm, which has proven extremely efficient for the numerical simulation of entropic optimal transport \cite{cuturi2013sinkhorn,chen2021stochastic}.


\subsection{Optimal transport plan in divergence-regularized OT}

To state the dual representation, for any $P\in {\rm Prob}(\Omega)$, we will need the functional $\Phi_{P}$ defined as
\begin{equation*}
	\Phi_{P} :\L^0 \longrightarrow  \R, \, \Phi_{P}(\xi) := \inf_{r \in \R}(\E^{P}[\ell^\ast(\xi - r)] + r),
\end{equation*}
	To alleviate the notation we will simply write $\Phi := \Phi_\P$.
This function is sometimes referred to as \emph{optimized certainty equivalent} in the literature. 
We refer the interested reader for instance to \cite{polyanskiy2025information,backhoff2020dynamic,bental2007old} for detailed accounts of properties of this function.


\begin{theorem}
\label{thm:exists.duality}
	Let {\rm \Cref{assump.data}} hold.
\begin{enumerate}[label=$(\roman*)$, ref=.$(\roman*)$,wide,  labelindent=0pt]
		\item If $\Pc_c(\muin,\mufin) \neq \emptyset$, then the problem \eqref{eq.SP} admits a unique optimizer $\Q^\star \in \Pc_c(\muin,\mufin)$.
		\item The problem \eqref{eq.SP} admits the convex dual representation
		\begin{equation}
		\label{eq:dualprobl.intro}
		V_c(\muin,\mufin) = \sup_{\varphi\in C_{b,p}(\R^m)}\Big(  \Psi^\varphi(\muin) - \langle \varphi,\mufin\rangle\Big) 
	\end{equation}
	with the functional $\Psi^\varphi$ being defined as
	\begin{align}\label{eq:defPsi}
	 \Psi^\varphi(\mu)\coloneqq -\int_{\R^\xdim} \Phi_{\P_x}\big(-Q_c\varphi(X_T) - C(X)\big) \mu(\d x) +\Ic_\ell(\mu|\nu_0), \mbox{ with, } Q_c\varphi(x) := \inf_{\rho \in \Pc_p(\R^\xdim)}(c(x, \rho) + \langle \varphi,\rho\rangle).
	\end{align}
	\end{enumerate}
\end{theorem}

\begin{remark}\label{rmk.duality.tv}
We remark for future reference that the case of bounded $c$ corresponds to the limiting case $p=\infty$ in $C_{b,p}(\R^m)$.
	In particular, in the case $c(x,\rho)=\int_{\R^\xdim} \1{x\neq y}\rho(\d y)$ associated to the terminal constraint $\Q\circ X_T^{-1}=\mufin$, thanks to {\rm \Cref{lemma.tv.ctransform}}, we have that $Q_c\varphi(x)=\varphi(x)$.
\end{remark}

In the rest of the paper we assume that $\Pc_c(\muin,\mufin) \neq \emptyset$.\medskip

Our next set of results studies the existence of dual optimizers, i.e., of Problem \eqref{eq:dualprobl.intro}, and provides characterizations for the density with respect to $\P$ of the optimal primal path measure $\Q^\star$ in terms of the dual optimizers.
	Because \eqref{eq:data.processing} fails for most divergences, we will consider the following notions:
	
\begin{definition}
	The divergence operator $\Ic_\ell$ is superadditive relative to $\P$ if for all $\Q\in\Pc(\muin)$
	\[
		\Ic_\ell(\Q|\P)\geq \Ic_\ell (\muin|\nu_0) + \int_{\R^\xdim}\Ic_\ell(\Q_x|\P_x) \muin(\d x),
	\]
	where $\Q_x(\d \omega)\muin(\d x)$ denotes the disintegration of $\Q$ with respect to $\Fc_0$.
	The divergence operator $\Ic_\ell$ is subadditive relative to $\P$ if the inequality ``$\ge$'' above is replaced by ``$\le$''.
\end{definition}

\begin{remark}
	The notions of super and subadditive divergence probably first explicitly appeared in the work of {\rm \citet*{lacker2018law}}.
	Interestingly, this author showed that beyond relative entropy, which is additive, i.e., both super and subadditive, other examples of functions $\ell$ leading to additive divergence are functions whose Legendre transforms are the exponential of a function of at most linear growth, see {\rm \cite[Remark 5.3]{lacker2018law}}. 
\end{remark}

When $\muin=\nu_0$, our existence and characterization of optimal plans will hold for essentially all divergence operators; otherwise, we will need to assume that $\ell$ is such that the associated divergence $\Ic_\ell$ is superadditive.
We collect these two cases under the following assumption: 
\begin{assumption}\label{assump.data.2}
$\P_{0T}\sim\muin\otimes\mufin$ and one of the following holds: $(i)$ $\muin=\nu_0$, $(ii)$ $\Ic_\ell$ is superadditive.
\end{assumption}

\begin{theorem}\label{thm.existence.dual}
	Let {\rm Assumptions \ref{assump.data}} and {\rm \ref{assump.data.2}} hold, fix $c(x,\rho) = \int \mathds{1}_{\{x \neq y\}}\rho(dy)$ and let $C(X)\equiv C(X_0,X_T)$.
	Then, {\rm Problem \eqref{eq:dualprobl.intro}} admits an optimal solution $\varphi^* \in B_{b,p}(\R^d)$.
	That is, there is $\varphi^*\in B_{b,p}(\R^d)$ such that
	\begin{equation*}
		V_c(\muin,\mufin) = \Psi^{\varphi^*}(\muin) - \langle\varphi^*, \mufin\rangle.
	\end{equation*}
\end{theorem}

\begin{remark}
	In addition to {\rm \Cref{assump.data.2}}, already commented above, {\rm \Cref{thm.existence.dual}} assumes that $c(x,\rho)=\int_{\R^\xdim} \1{x\neq y}\rho(\d y)$, that is, recall {\rm \Cref{rmk.duality.tv}}, the terminal constraint corresponds to $\Q\circ X_T^{-1}=\mufin$.
	This is done to bypass technicalities arising from the term $Q_c\varphi$ in the dual formulation \eqref{eq:dualprobl.intro}.
	Let us also mention that the assumption $C(X)\equiv C(X_0,X_T)$ arises, ultimately, as we rely on the results of {\rm\citet*{ruschendorf1993note}}.
	We leave the general as the subject of further research and comment on the steps needed to overcome the particular choice of weak cost after the proof of {\rm \Cref{thm.existence.dual}} in {\rm \Cref{sec.dual.optimizers}}.
\end{remark}

We are now in a position to present the announced characterization of the density with respect to $\P$ of the optimal primal path measure $\Q^\star$ in terms of the dual optimizers.

\begin{theorem}
\label{thm:propertyMain}
Let {\rm Assumptions \ref{assump.data}} and {\rm \ref{assump.data.2}} hold. Let $\varphi^\star \in B_{b,p}(\R^\xdim)$ 
be optimal for the dual problem \eqref{eq:dualprobl.intro} satisfying $\E^{\muin\otimes \P_\cdot}[\ell^\ast((Q_c\varphi^\star (X_T) + C(X))^+)]<+ \infty$. There is a measurable function $\psi^\star:\R^m\longrightarrow \R$ such that 
	\begin{equation*}
		\dfrac{\d\Q^\star\!}{\d \P} =\frac{\d\muin }{\d\nu_0} (X_0)   \partial_x\ell^*\Big(  - Q_c\varphi^\star(X_T) - C(X) -\psi^\star(X_0)\Big).
	\end{equation*}
\end{theorem}

\begin{remark}
	Let us comment on the above result. 
	First, note that the existence of a dual optimizer is complemented by the integrability condition $\E^{\muin\otimes \P_\cdot}[\ell^\ast((Q_c\varphi^\star (X_T) +C(X))^+)]<\infty$.
	In the case $\muin=\nu_0$, the condition reduces to $\E^{\P}[\ell^\ast((Q_c\varphi^\star (X_T) + C(X))^+)]<\infty$ and guarantees a dual representation for the functional $\Phi_{\P_x}$, for $\mu_0\ae\; x\in \R^\xdim$, see {\rm \Cref{prop:dual.f.div}}.
	Second, the reader might find it illustrative to note that in the case $\ell(x)=x\log(x)-x+1$ and $c(x,\rho)=\int_{\R^\xdim} \1{x\neq y}\rho(\d y)$, we have that $\ell^\ast(x)={\rm exp}(x)$ and $Q_c\phi(x)=\phi(x)$.
	In this case, we recover the well-known characterization of the optimal entropic optimal transport plan.
	That is, 
	\begin{equation*}
		\frac{\d\Q^\star}{\d\P} = \frac{\d\muin}{\d\nu_0} \e^{-C(X)} \e^{-\varphi^\star(X_T) -\psi^\star(X_0)}.
	\end{equation*}
\end{remark}

We will now state a verification theorem.
In the statement below, we use the time reversal mapping $\overleftarrow{\Tc}:\Omega\longrightarrow \Omega$ given by $\overleftarrow{\Tc}(\omega):= \overleftarrow{\omega}$ where $\overleftarrow{\omega}_t:= \omega_{T-t}$ and the time-reversed measure $\overleftarrow{\P} = \P\circ \overleftarrow{\Tc}$.
	We distinguish two cases in the following assumption:
	
\begin{assumption}\label{assump.data.3}
{\color{black} $\P_{0T}\sim\muin\otimes\mufin$ and one of the following holds: $(i)$ $\muin=\nu_0$, $(ii)$ $\Ic_\ell$ is subadditive}.
\end{assumption}

\begin{theorem}
\label{thm:property.main.2}
	Let {\rm Assumptions \ref{assump.data} and \ref{assump.data.3}} hold, and fix $c(x,\rho)=\int_{\R^\xdim} \1{x\neq y}\rho(\d y)$. 
	
	Let $\hat \varphi \in B_{b,p}(\R^\xdim)$ satisfy $\E^{\muin\otimes \P_\cdot}[\ell^\ast(( \hat \varphi  (X_T)+ C(X))^+)]< + \infty$, $\hat\psi$ be measurable, and let $\Q\in \Pc_c(\muin,\mufin) $ be given by
	\begin{equation*}
			\dfrac{\d \Q}{\d \P} =\frac{\d\muin }{\d\nu_0} (X_0)   \partial_x\ell^*\big(  - \hat \varphi(X_T) - C(X) -\hat \psi(X_0)\big).
	\end{equation*}
	Then it holds:
\begin{enumerate}[label=$(\roman*)$, ref=.$(\roman*)$,wide,  labelindent=0pt]
	\item  The probability measure $ \Q$ is primal optimal and the function $\hat \varphi$ is dual optimal.
	\item  The functions $\hat\varphi$ and $\hat\psi$ satisfy the Schr\"odinger system
	\begin{align}\label{eq.prop.system}
	\begin{cases}
		1= \E^{\P} \big[  \partial_x\ell^*\big(  - \hat \varphi(X_T) - C(X) -\hat \psi(X_0)\big) |X_0=x\big], \; \muin\ae\\
		1= \E^{\overleftarrow \P} \big[  \partial_x\ell^*\big(  - \hat \varphi(\overleftarrow X_T) - C(\overleftarrow X) -\hat \psi(\overleftarrow X_0)\big) |\overleftarrow X_0=x\big], \;\mufin\ae
	\end{cases}
	\end{align}
\end{enumerate}
\end{theorem}

\begin{remark}
\label{rem:entropic-Schroedinger}
	When the measure $\P$ is reversible, i.e. when $\overleftarrow{\P} = \P$, the system \eqref{eq.prop.system} becomes
	\begin{align*}
	\begin{cases}
		1= \E^{\P} \big[  \partial_x\ell^*\big(  - \hat \varphi(X_T) - C(X) -\hat \psi(X_0)\big) |X_0=x\big], \; \muin\ae\\
		1= \E^{ \P} \big[  \partial_x\ell^*\big(  - \hat \varphi( X_T) - C( X) -\hat \psi( X_0)|X_T=x], \;\mufin\ae
	\end{cases}
	\end{align*}
	To keep the paper self-contained we give the argument after the proof of {\rm \Cref{thm:property.main.2}} below.
	In fact, in the entropic case and when $\P$ is reversible, the system \eqref{eq.prop.system} reduces to the ``classical'' Schrödinger system. 
	In fact if $C(X) = \tilde C(X_0,X_T)$ for some function $\tilde C$ on $\R^\xdim\times \R^\xdim$, the equations become
	\begin{equation*}
	 	\begin{cases}
	 		\hat\psi(x) = -\log\E\big[\e^{-\hat\varphi(X_T) - \tilde C(x,X_T)}|X_0=x],\; \muin \as\\
	 		\hat\varphi(x) = -\log\E\big[\e^{-\hat\psi(X_0) - \tilde C(X_0,x)}|X_T=x],\; \mufin \as
	 	\end{cases}
	 \end{equation*} 
	derived, e.g., by {\rm \citet{leonard2014survey}}.
	This is not the only case of interest.
	As we will comment in {\rm \Cref{sec:examples}}, the case of the $\chi^2$-divergence \eqref{eq.prop.system} results in a tractable system of equations as well.
	In general, the system remains amenable to regression-type methodologies.
\end{remark}



\subsection{Flows of marginals for non-entropic dynamic Schrödinger bridges}

We now turn our attention to describing the flow of marginal measures $(\Q_t^\star)_{t\in [0,T]}$, where $\Q^\star_t\coloneqq \Q^\star\circ X_t^{-1}$ denotes the marginal law of $X$, induced by $\Q^\star$, the optimal solution to the non-entropic Schrödinger problem.
	This, of course, requires the SDE \eqref{eq.dyn.P} to be Markovian with reversible law $\P$.
Our first result provides a characterization in terms of Markovian projections, in the spirit of \citet{follmer1988random}.

\begin{theorem}
\label{them:charac.marginal}
	Let {\rm Assumptions \ref{assump.data}} and {\rm \ref{assump.data.2}} hold. 
	Assume that $C=0$, $c$ is such that $\Wc_c(\mu,\nu)=0$ if and only if $\mu = \nu$, that, for $(t,x)\in [0,T]\times C([0,T],\R^\xdim)$, $b(t,x) = -\partial_x U(x(t))/2$ for some continuously differentiable function $U:\R^\xdim\longrightarrow \R$, and let $\nu_0$ be the invariant measure on \eqref{eq.dyn.P}.
	Further assume that $V_c(\muin,\mufin)$ and $V_c(\mufin,\muin)$ admit dual optimizers $\varphi^\star \in B_{b,p}(\R^m)$ and $\overleftarrow{\varphi}\in B_{b,p}(\R^m)$ respectively.
Then there are two predictable processes $\alpha^\star$ and $\overleftarrow{\alpha}$ such that 
 	\begin{equation}
 	\label{eq:non-markovian.charac.statement}
		\E^{\Q^{\star}}\Big[ \alpha^\star_t  + \overleftarrow{\alpha}_{T-t} \circ \overleftarrow{\Tc}  \,  \big |X_t = x \Big]  -\partial_x U(x)= \partial_x \log(\Q^{\star}_t(x)),\;  {\color{black} \d t\otimes \d\Q_t^\star\ae}~\text{on } [0,T]\times\R^\xdim,
	\end{equation}
	where $\Q^\star_t$ is the probability density function of $\Q^\star\circ X_t^{-1}$, and {\color{black} where the derivative is in the sense of distributions}.
\end{theorem}
The above characterization result can be further specialized, depending on the choice of the function $\ell$.
We illustrate this with two corollary concerning popular choices of divergences.
\begin{corollary}[Relative entropy]
\label{cor:rel.ent.cha}
Let the assumptions of {\rm \Cref{them:charac.marginal}} hold and let $\ell(x) = x\log x - x+ 1$. Then we have:

\begin{enumerate}[label=$(\roman*)$, ref=.$(\roman*)$,wide,  labelindent=0pt]
	\item 	The {\rm \Cref{eq:non-markovian.charac.statement}} reduces to 
	\begin{equation*}
		f_t(x)   + g_{T-t}(x)  -\partial_x U(x ) = \partial_x \log(\Q^{\star}_t(x)),\;  \d t\otimes \d\Q_t^\star\ae~\text{on } [0,T]\times\R^\xdim,
	\end{equation*}
	for two Borel-measurable functions $f,g:[0,T]\times \R^\xdim\longrightarrow  \R^\xdim$.
	\item If in addition the functions $\partial_xU, Q_c\varphi^\star$ and $\overleftarrow{\varphi}$ are continuously differentiable. 
	Then, $f$ and $g$ satisfy $f_t(x) = \partial_x v^\star(t,x)$ and $g(t,x) = \partial_x \overleftarrow{v}(t,x)$ where $v$ and $\overleftarrow{v}$ are solutions of the {\rm HJB} equation
	\begin{equation*}
	 	\partial_t v(t,x) - \partial_xv(t,x)\cdot  \partial_x U(x)  +\frac12\partial_{xx}v(t,x) -\frac12|\partial_xv(t,x)|^2 = 0, \, (t,x)\in [0,T)\times\R^\xdim,\\ 
 	 \end{equation*} 
 	 with respective terminal conditions $v^{\star}(T,x) = Q_c\varphi^\star(x)$ and $\overleftarrow{v}(T,x) = \overleftarrow{\varphi}(x)$.
 	\end{enumerate}
\end{corollary}

\begin{remark}
	Some remarks are in order concerning the above characterization of the marginals of the Schr\"odinger bridge:
	Based on the fact that the potentials $f$ and $g$ are also gradients, the characterization in {\rm \Cref{cor:rel.ent.cha}}
	becomes the well-known formula
	\begin{equation*}
		\Q_t^\star(x)  =  \exp(v^\star(t,x))\exp(\overleftarrow{v}(T-t,x))\exp(U(x) ).
	\end{equation*}
	
	We do not expect this property to extend to other divergence operators beyond the relative entropy.
	This is because the functions $f$ and $g$ are constructed through solutions of stochastic optimal control problems $($given in \eqref{eq:cont.problem}$)$.
	By verification theorems, such processes are constructed as minimizers of a Hamiltonian.
	The said Hamiltonian would typically not have closed-form minimizers.
	Nonetheless, we can obtain pointwise characterizations in some non-entropic cases, see, e.g., {\rm\Cref{cor:quad.cha}} below.

	Moreover, {\rm \Cref{them:charac.marginal}} also assumes that $\Wc_c$ is a distance.
	This assumption is important to relate the values $V_c(\muin,\mufin)$ and $V_c(\mufin,\muin)$, which is a key element of the proof.
	We believe that this assumption cannot be dropped.
\end{remark}

As discussed in the introduction, quadratic regularized optimal transport has been an active area in recent years, notably due to the sparsity property of optimal transport plans, see, e.g., \cite{gonzalez2025sparse,gonzalez2025sample,nutz2025quadratically,gonzalez2024sparsity,zhang2023manifold,wang2025quadratic,wiesel2025sparsity}.
	In the dynamic setting this corresponds to the choice of $\chi^2$-divergence, the flow of marginals of which the following result characterizes:
\begin{corollary}($\chi^2$-divergence)
\label{cor:quad.cha}
	Let the assumptions of {\rm \Cref{them:charac.marginal}} hold and let $\ell(x) = (x -1)^2/2$ if $x\ge0$ and $\infty$ else.
	Further assume that the functions $\partial_xU, Q_c\varphi^\star$ and $\overleftarrow{\varphi}$ are twice continuously differentiable with bounded derivatives. 
	Let $v$ and $\overleftarrow{v}$ solve the {\rm PDE}s
	\begin{equation}
	\label{eq:PDE.prop.quadratic1.intro}
	\begin{cases}
	 	\partial_t v(t,x) - \partial_xv(t,x)\cdot \partial_x U(x)   +\frac12\partial_{xx}v(t,x) -\frac12|\partial_x\tilde v(t,x)|^2 = 0, \, (t,x)\in [0,T)\times\R^\xdim,\\ 
	 	 v(T,x) = 0\\ 
	 	\partial_t \tilde v(t,x) - \partial_x\tilde v(t,x)\cdot \partial_x U(x)   +\frac12\partial_{xx}\tilde v(t,x)  = 0, \, (t,x)\in [0,T)\times\R^\xdim,\\ 
	 	 \tilde v(T,x) = \zeta(x)
	 	\end{cases}
 	\end{equation} 
 	with respective terminal conditions $\zeta(x) = Q_c\varphi^\star(x)$ and $\zeta(x) = \overleftarrow{\varphi}(x)$.
 	Let $Z$ be given by 
	\begin{equation*}
 		\d Z_t = -\partial_xv(t, X_t)\d W_t, \quad Z_0 =1.
 	\end{equation*}
	Then the density $\Q^\star_t$ of $\Q^\star\circ(X_t, Z_t)^{-1}$  exists and satisfies
 	\begin{equation}\label{flow.chisquare}
 		\Fc(t,x,z) + \overleftarrow  \Fc(T-t,x,z) = \Sigma\Sigma^\top(t,x,z)\nabla   \log(\Q^\star_t(x,z)) ,\;  {\color{black} \d t\otimes \d\Q_t^\star\ae}~\text{on } [0,T]\times\R^\xdim\times \R,
 	\end{equation}
 	where the functions $\Fc,\overleftarrow \Fc$ and $\Sigma$ are defined as
 	\begin{equation*}
 		\Fc(t,x,z) = 
		\begin{pmatrix}
			- \frac{\partial_xv(t,x)}{z} - \frac{\partial_x U(x)}2 \\
 		\frac{|\partial_xv(t,x)|^2}{z}+\frac{\partial_{xx} v(t,x)}2
		\end{pmatrix},\; 
 		\overleftarrow  \Fc(t,x,z) = 
		\begin{pmatrix}
			-\frac{\partial_x\overleftarrow  v(t,x)}{z} -\frac{ \partial_x U(x)}2  \\
 			\frac{|\partial_x\overleftarrow  v(t,x)|^2}{z} +\frac{\partial_{xx} v(t,x)}2
		\end{pmatrix},
		\text{ and, }
 		\Sigma(t,x,z) = 
			\begin{pmatrix}
			I_\xdim  \\
 			\partial_x v(t,x)
			\end{pmatrix}
	\end{equation*}
	where $I_\xdim$ is the identity matrix of $\R^{\xdim\times \xdim}$.
\end{corollary}

\subsection{Examples}
\label{sec:examples}
Let us conclude this section with some examples
of possible choices of divergences and weak costs $c$.
\paragraph*{Divergence.}
Beyond the classical case of relative entropy $\ell(x)=x\log x-x + 1$ discussed above, one could also consider the following choices of divergence for which we illustrate the form of our results.

\begin{enumerate}[label=$(\roman*)$, ref=.$(\roman*)$,wide,  labelindent=0pt]
\item \emph{$\chi^2$-divergence}. This corresponds to $\ell(x) = (x-1)^2/2$ if $x\ge0$ and $+\infty$ else.
In this case, the convex conjugate is $\ell^\ast(x) =x + x^2/2 $ for $x\ge -1$ and $\ell^\ast(x)=-1/2$ else.
In this case, the primal optimizer $\Q^\star$ is given by
\begin{equation*}
	\frac{\d\Q^\star}{\d \P} = \frac{\d \muin}{\d \nu_0}(X_0)  \Big(1- Q_c \varphi^\star(X_T) - C(X) - \psi^\star(X_0) \Big)^+
\end{equation*}
where $\varphi^\star$ is a dual optimizer and $x^+$ denotes the positive part function; and if {\color{black} $c(x,\rho)=\int_{\R^\xdim} \1{x\neq y}\rho(\d y)$}, 
 the functions $\varphi^\star$ and $\psi^\star$ satisfy the system
	{\color{black}
	\begin{equation*}
	\begin{cases}
		1 = \E^{\P} \big[ \big(1 -  \varphi^\star(X_T) - C(X) - \psi(X_0) \big)^+ |X_0=x\big], \; \muin\ae\\
		1= \E^{\overleftarrow \P} \big[ \big( 1 -   \varphi^\star(\overleftarrow X_T) - C(\overleftarrow X) - \psi^\star(\overleftarrow X_0)\big)^+ |\overleftarrow X_0=x\big], \;\mufin\ae
	\end{cases}
	\end{equation*}}

\item  \emph{Tsallis entropy.} Given in terms of $\ell(x)=\frac{x^q-1}{q-1}$, $x\geq 1$ and $+\infty$ else, for $q\neq 1$. 
	This family encompasses a rather large class of $f$-divergences that fits our setting for $q > 1$.
	The convex conjugate is $\ell^\ast(x) =\frac{1}{q-1}+ (\frac{q-1}{q} x)^{q/(q-1)}$ for $x\geq0 $ and $\ell^\ast(x)=\frac{1}{q-1}$ else.
	In this case, the primal optimizer $\Q^\star$ is given by 
\begin{equation*}
	\frac{\d\Q^\star}{\d \P} =\Big(  \frac{q-1}{q}\Big)^{\frac1{q-1}}  \frac{\d \muin}{\d \nu_0}(X_0) \Big(  \big(- {\color{black} Q_c\varphi^\star(X_T) }- C(X) - \psi^\star(X_0) \big)^+\Big)^{\frac1{q-1}}.
\end{equation*}

\item \emph{Squared-Hellinger divergence}. Defined by $\ell(x) = (1-\sqrt{x})^2$ if $x\ge0$ and $+\infty$ else.
	This choice offers a slight variation of the previous example.
	Though it does not satisfy the growth assumption at infinity, we illustrate the form of our results. 
	The convex conjugate is $\ell^\ast(x) =x/(1-x)$ for $x< 1$ and $\ell^\ast(x)=\infty$ else.
	In this case, the primal optimizer $\Q^\star$ is given by
\begin{equation*}
	\frac{\d\Q^\star}{\d \P} =  \frac{\d \muin}{\d \nu_0}(X_0)  \frac{1}{\big\{\big(1- Q_c \varphi^\star(X_T) - C(X) - \psi^\star(X_0) \big)^+\big\}^2}.
\end{equation*}

\end{enumerate}
%
%
%
%


\paragraph*{Weak transport costs.}
Let us first note that the weak OT problem \eqref{eq.weaktransport} admits the probabilistic interpretation
\begin{align*}
 \Wc_c (\mu,\nu)=\inf_{X\sim\mu,Y\sim \nu} \E[ c(X,\E[Y|X])],
\end{align*} 
revealing how, though initially motivated by applications to geometric inequalities, weak optimal transport theory highlights the relevance and implications of imposing structural constraints on feasible plans beyond fixed marginals.\medskip

We now detail some instances of costs $c$ and elucidate their associated relation $\Q_T\coloneq \Q\circ X_T^{-1}\overset{c}{=}\mufin$ induced by $\Wc_c (\Q_T,\mufin)=0$. 
	We refer the reader to \cite{gozlan2020mixture,gozlan2017kantorovich} for a detailed account of these instances.

\begin{enumerate}[label=$(\roman*)$, ref=.$(\roman*)$,wide,  labelindent=0pt]
\item \emph{Total variation}. As mentioned above, when $c(x,\rho)=\int_{\R^\xdim} \1{x\neq y}\rho(\d y)$, $\Wc_c$ is  the total variation distance and $\Q_T \overset{c}=\mufin$ if and only if $\Q_T = \mufin$ and in this case $Q_c\varphi = \varphi$. 

\item \emph{Marton’s cost.} Takes its name from the work \cite{marton1996bounding} and further studied in \cite{gozlan2017kantorovich}. The Marton’s cost functional is given by
\( 
c(x, \pi_x ) = \theta \big(\int_{\R^\xdim} \delta(\textup{c}(x, y))\pi_x (\d y)\big),
\)
where $\textup{c}:\R^\xdim\times\R^\xdim\longrightarrow \R\cup\{+\infty\}$ is a classical cost,
$\delta: \R_+ \longrightarrow \R_+$ is l.s.c.~and $\theta :\R_+\longrightarrow [0,\infty]$ is convex. 
	It has the probabilistic interpretation
\begin{align*}
\Wc_{c}  (\Q_T,\mufin ) =\inf_{X_T\sim\Q_T ,Y\sim\mufin }\E \big[ \theta  \big( \E\big[ \delta (c (X_T,Y))|X_T ]\big) \big].
\end{align*} 
	The classical transport problem with cost $\textup{c}$ is recovered with the choices $\theta(x) = \delta(x) = x$.
	In this case, $Q_c\varphi(x) =  \varphi^c(x)\coloneqq \inf_{y\in \R^\xdim}\big\{ \varphi(y)+\| x-y\|^p\big\}$, is the $c$-transform of $\varphi$, see \cite[Chapter 5]{villani2009optimal}.
	In particular, if $\textup{c}(x,y) = \|x - y\|^p$ and $\textup{c}(x,y) = \1{x\neq y}$ then $\Wc_c$ become the $p$-Wasserstein and the total variation, respectively.
	In this context, Brenier's Theorem \cite{benamou2000computational} shows that if $\textup{c}\equiv \|x-y\|^p, p>1$ and $\Q_T\ll {\rm Leb}$, then $\Q_T \overset{c}=\mufin$ implies there exists some convex function $\phi$ such that $\nabla \phi$ pushes forward $\Q_T$ onto $\mufin$.
	
\item \emph{Barycentric cost.} The Barycentric cost corresponds to 
\(
c(x, \pi_x ) = \theta\big( x-  \int_{\R^\xdim}y \pi_x (\d y) \big),
\)
 for $\theta:\R^\xdim \longrightarrow \R_+$ convex and l.s.c. That is, we have probabilistic interpretation
\begin{align*}
\Wc_{c}  (\Q_T,\mufin ) =\inf_{X_T\sim\Q_T ,Y\sim\mufin }\E \big[ \theta\big ( X_T-\E[Y|X_T]\big) \big].
\end{align*} 
	This choice captures the existence of \emph{martingale couplings} between two measures.
	That is, $\Q_T\overset{c}=\mufin$ ensures that there exists a martingale coupling between $\Q_T$ and $\mufin$.
	In this case, $Q_c\varphi(x)=Q_\theta \varphi_{cvx}(x)$, where $\varphi_{cvx}$ denotes the greatest convex function $\varphi_{cvx}:\R^\xdim \longrightarrow \R$, such that $\varphi_{cvx}\leq \varphi$ and $ Q_\theta \varphi\coloneqq \inf_{y\in \R^\xdim}\big\{ \varphi(y)+\theta(x-y)\big\}$, see \cite{gozlan2020mixture, backhoff2019existence} for further details. 

\item \emph{Distributionally robust costs.} This weak cost builds on the \emph{variance case with a quadratic cost studied} in \citet[Section 5.2]{alibert2019new}, the distributionally robust cost takes the form 
\[
c (x,\pi_x  )=  \int_{\R^\xdim} \textup{c} (x,y)\pi_x (\d y)  - \lambda^{-1}  {\rm Var}(\pi_x), \;\text{ where, }  {\rm Var}(\rho )\coloneqq  \int_{\R^\xdim} \Big\|y-\int_{\R^\xdim}y^\prime\rho  (d y^\prime)\Big\|^2  \rho  (\d y),
\]
	and $\lambda>0$, 
	acts as an interpolation parameter between the classical cost $\textup{c}$ and the variance of $\pi_x$, thus favouring the spreading of the measures $\pi_x$.
	Thanks to \cite[Lemma 5.5]{alibert2019new}, for any $r>0$, with $\lambda:={\rm Var}(\nu)/r^2$ and ${\rm c}(x,y)=\|x-y\|^2$ it holds that
	\[
	\Wc_{c} (\mu ,\nu) +r^2 \geq  W_2^2(\mu,\nu).
	\]
	Thus, $\Q_T \overset{c}= \mufin$ implies $ \Q_T \in B_{W_2} (\mufin, r)$,
	where $B_{W_2}(\nu, r )$ is the Wasserstein-2 ball of radius $r>0$ centered at $\nu\in\Pc_2(\R^\xdim)$. {\color{black} In this case, $Q_c\varphi (x) = \inf_{y\in\R^\xdim} \big\{\varphi(y) + \frac{\lambda}2 \|x-y\|^2\big\}$ is the Moreau-Yosida transform of $\varphi$.} 
\end{enumerate}	

\section{Preliminaries}
\label{sec:preliminaries_and_examples}
Before presenting the proofs of the main results, we gather a few technical lemmas that will be used in the proofs. 
We first establish some convexity properties.
\begin{lemma}\label{lemma.convexcontrols}
	Let {\rm \Cref{assump.data}} hold. 
	The set $\Pc_c(\muin,\mufin)$ and the mapping $(\muin,\mufin)\longmapsto V_c(\muin,\mufin)$ are convex.
\end{lemma}
\begin{proof}
	We begin by convexity of $\Pc_c(\muin,\mufin)$.
	Let $\overline \Q = \lambda \Q^1 + (1 - \lambda)\Q^2$ for $\Q^1,\Q^2\in \Pc_c(\muin,\mufin)$.
	It is clear that $\overline \Q \circ X^{-1}_0=\muin$, thus $\bar \Q	\in \Pc(\muin)$.
	Recall from \cite[Theorem 1.3]{backhoff2019existence} the duality formula for $\Wc_c$
	\begin{align}\label{eq.dualitywot}
		\Wc_c(\mu , \mufin) = \sup_{\varphi \in C_{b,p}(\R^d)}(\langle Q_c\varphi,\mu \rangle -\langle \varphi,\mufin\rangle  ).
	\end{align}	
	Note that $\overline \Q \circ X^{-1}_T=\lambda \Q^1\circ X^{-1}_T+(1-\lambda) \Q^2\circ X^{-1}_T$, and, since $c$ is nonnegative, it follows that 
	\[
	0\leq \Wc_c(\overline \Q\circ X^{-1}_T,\mufin ) \le \lambda \Wc_c(\Q^1\circ X^{-1}_T, \mufin) + (1-\lambda)\Wc_c(\Q^2\circ X^{-1}_T, \mufin ) = 0.
	\]
	Hence, $\bar \Q\in \Pc(\muin, \mufin)$.

	\medskip

	Regarding convexity of $V_c$, let $\lambda\in (0,1)$, $\muin^i,\mufin^i\in \Pc(\R^d)$, $i=1,2$ be given, and $\Q^i \in \Pc_c(\muin^i,\mufin^i)$.
	Let $\overline \Q = \lambda \Q^1 + (1 - \lambda)\Q^2$ and define ${\overline\mu}_0$ and ${\overline \mu}_T$ similarly.
	We claim that $\overline \Q\in \Pc_c({\overline\mu}_0,{\overline\mu}_T)$.
	In fact, $\overline \Q\circ X^{-1}_0={\overline\mu}_0$ is clear.
	Furthermore,
	 $\overline \Q\circ X^{-1}_T=\lambda \Q^1\circ X^{-1}_T+(1-\lambda) \Q^2\circ X^{-1}_T$, and since $c$ is nonnegative, it follows from \eqref{eq.dualitywot} that 
	\[
	0\leq \Wc_c(\overline \Q\circ X^{-1}_T,{\overline \nu}_T ) \le \lambda \Wc_c(\Q^1\circ X^{-1}_T, \mufin^1) + (1-\lambda)\Wc_c(\Q^2\circ X^{-1}_T, \mufin^2 ) = 0.
	\]

	Therefore, by convexity of $\ell$,
	\(
		V_c({\overline\mu}_0,{\overline\mu}_T)  \le\E^{\overline{\Q}}[C(X)] + \Ic_\ell\big(\overline \Q  |\P\big) \leq  \lambda\{\E^{\Q^1}[C(X)] + \Ic_\ell\big( \Q^1  |\P\big)\} + (1 - \lambda)\{\E^{\Q^2}[C(X)] +\Ic_\ell\big( \Q^2  |\P\big)\} .
	\)
	Since $\Q^1$ and $\Q^2$ were taken arbitrary we deduce that $V_c({\overline\mu}_0,{\overline\mu}_T) \leq \lambda  V_c(\muin^1,\mufin^1)  +(1-\lambda)V_c(\muin^2,\mufin^2) $.\qedhere
\end{proof}

\begin{proposition}
\label{prop.v.lsc}
	Let {\rm \Cref{assump.data}} hold.
	The function $\Pc(\R^d)\times \Pc(\R^d)\ni (\mu,\nu)\longmapsto V_c(\mu,\nu) \in \R \cup\{+\infty\}$ is lower semicontinuous for the topology of weak convergence.
\end{proposition}
\begin{proof}
	Let $(\muin^n,\mufin^n)$ be a sequence of elements of $\Pc(\R^d)\times \Pc(\R^d)$ converging to $(\muin ,\mufin)$.
	Without loss of generality, we assume that $\liminf_{n\to \infty}V_c(\muin^n,\mufin^n) <\infty$.
	Thus, up to a subsequence, $(V_c(\muin^n,\mufin^n))_{n\in \N}$ is bounded.
	In particular, since $C$ is bounded from below, there is a constant $K>0$ such that for each $n\ge1$ there is $\Q^n\in  \Pc_c(\muin^n,\mufin^n)$ satisfying,
	\begin{align}
	\label{eq:lsc.eps-optimal}
		K>V_c(\muin^n,\mufin^n) - \E^{\Q^n}[C(X)]
		 \geq 
		   \E\big[\ell\big(Z^n \big)\big]
		  - 
		  \frac1n,\; 
		  \text{where } 
		  Z^n\coloneqq \dfrac{\d\Q^n}{\d\P}. 	
	\end{align}

Thus, the sequence $(Z^n)_{n\geq 1}$ is uniformly integrable and, consequently, tight.
\medskip
	
We claim that any of the weak limits of $(Z^n)_{n\geq 1}$ gives rise to a probability measure $\Q\in\Pc_c(\muin,\mufin)$ with $\Q\ll \P$.
	Let $Z$ be a weak limit of $(Z^{n_k})_{k\geq 1}$ and consider the process $M_t\coloneqq \E[ Z|\Fc_t]$, $t\in [0,T]$.
	The convexity of $\ell$ shows that $\sup_{t\in [0,T]} \E [\ell(M _t)]\leq  \E [\ell(Z)]<\infty$, and, by the growth assumption on $\ell$, $M$ is a uniformly integrable martingale.
	Since, $\E[Z^n]=1$, $n\geq 1$, we deduce from the convergence in law of $(Z^{n_k})_{k\geq 1}$ that $\E[Z]=\E[M_t]=1$, $t\in [0,T]$.
	With this, it follows that $\Q(A)\coloneqq \E[M_T \cdot \mathds{1}_{A}]$ is a well defined probability measure and $\Q\ll \P$.
	\medskip
	
To prove the claim it remains to show that $\Q\in\Pc_c(\muin,\mufin)$. Recall that $\Q^n\in  \Pc_c(\muin^n,\mufin^n)$, i.e., $\Q^{n}\circ X^{-1}_0=\muin^n$, $\Q^{n}\circ X^{-1}_T\=c\mufin^n$ for all $n\geq 1$ and $(\muin^n,\mufin^n)$ converges in law to $(\muin ,\mufin)$.
	Thus, for a bounded continuous function $f$
	\[
	\E^{\Q\circ X_0^{-1}}[f] =\E^\Q[f(X_0)]=\lim_{k\to \infty} \E^{\Q^{n_k}}[f(X_0)]=\lim_{k\to \infty} \E^{\muin^{n_k}}[f ]= \E^{\muin}[f],
	\]
	and similarly, the lower semicontinuity of $\mu\longmapsto \Wc_c( \mu,\mu_T)$, gives
	\[
	\Wc_c( \Q\circ X_T^{-1},\mu_T)\leq \liminf_{k\to \infty}\Wc_c( \Q^{n_k}\circ X_T^{-1},\mu_T)=\liminf_{k\to \infty}\Wc_c( \mufin^{n_k},\mu_T)=0.
	\]
	
We now conclude. Back in \eqref{eq:lsc.eps-optimal} we see that
	\[
	\liminf_{n\longrightarrow\infty} V(\muin^n,\mufin^n)\geq  \liminf_{n\longrightarrow\infty} \E[Z^nC(X)]+ \E[ \ell(Z^n)]\geq \E[M_TC(X)] + \E[ \ell(M_T)] \geq V(\muin,\mufin).\qedhere
	\]
\end{proof}

It will be useful to see the regularized optimal transport problem \eqref{eq.SP} as an optimal stochastic control problem.
To this end we show in the following lemma that elements of $\Pc(\muin,\mufin)$ can be written as stochastic exponential of local martingales.
Let $\H^{2}_{\rm loc}$ be the space of $\F$--predictable processes $\alpha$ such that $\P \big(\int_0^T\|\alpha_t\|^2\d t<\infty\big)=1$.

\begin{lemma}\label{lemma.ref}
Let {\rm \Cref{assump.data}} hold. 
Let
\begin{align*}
		\Ac(\muin) \coloneqq  \bigg\{\P^\alpha \in {\rm Prob}(\Omega):  \frac{\d\P^\alpha\!\!}{\d\P\hfill} =\1{\frac{\d\P^\alpha\!\!}{\d\P\hfill}>0}  Z_T,\, \d Z_t=\alpha_tZ_t\d W_t,\,{\color{black} Z_0=\frac{\d\muin }{\d\nu_0} (X_0)} ,\, \P\as, \, \alpha\in \H^{2}_{\rm loc},\, \P^\alpha\circ X_0^{-1}=\muin  \bigg\}
\end{align*}
and $J(\P^\alpha)\coloneqq\displaystyle \frac12\E^{\P^\alpha}\bigg[ \int_0^T \ell^{\prime\prime}(Z_t)Z_t \|\alpha_t\|^2 \d t\bigg]+\Ic_\ell(\muin|\nu_0)$.\footnote{$\Ac_c(\muin,\mufin)$ is defined analogous to $\Pc_c(\muin,\mufin)$.}
Then the following holds:
\begin{enumerate}[label=$(\roman*)$, ref=.$(\roman*)$,wide,  labelindent=0pt]
	 \item $\{\P^\alpha\in \Ac(\muin): J(\P^\alpha)<\infty\} = \{\Q\in \Pc(\muin): \Ic_\ell(\Q|\P)<\infty\}$ and $\Ic_\ell(\P^\alpha|\P) = J(\P^\alpha)$;
	 \item $V_c(\muin,\mufin) = \inf_{\P^\alpha\in \Ac_c(\muin,\mufin)}(\E^{\P^\alpha}[C(X)] + J(\P^\alpha) )$.
\end{enumerate}


\begin{proof}
{\rm $(i)$}	Let $\Q\in\Pc(\muin)$ be such that $\Ic_\ell(\Q|\P)<\infty$.
 	Then $\Q\ll\P$, so that the density $\frac{\d\Q}{\d\P}$ exists.
	Letting $M_t\coloneqq \E [ Z |\Fc_t ]$, $Z\coloneqq \tfrac{\d\Q}{\d\P}$, the convexity of $\ell$ shows that $\sup_{t\in [0,T]} \E [\ell(M _t)]\leq  \E [\ell(Z)]=\Ic_\ell(\Q|\P)<\infty$, and thus $M$ is a uniformly integrable martingale.
	Let us now introduce 
	\[
	\tau_n\coloneqq\inf \{t\geq 0: M_t\leq 1/n\}, 
	\text{ and }
	\tau \coloneqq\inf \{t\geq 0: M_t=0\}.
	\]
	Note that, $\tau_n\leq \tau_{n+1}\leq \tau, \P\as$ and $\tau_n \longrightarrow \tau,\, \P\as,$ as $ n\longrightarrow \infty$.
	Let us put $M^n_t\coloneqq M_{t\wedge \tau_n}$ and note that, since $M^n>0$, it follows from It{\^o}'s formula that
	\[
	M^n_t=\exp\Big( L_t^n-\frac12\langle L^n\rangle_t\Big), \; L_t^n\coloneqq \log(M_0^n)+ \int_0^t \frac{1}{M_s^n}\d M_s^n.
	\]
	Passing to the limit we find that, for $L_t\coloneqq \log(M_0) + \int_0^t \frac{1}{M_s}\d M_s$,
	\[
	M_t=\exp\Big( L_t-\frac12\langle L\rangle_t\Big), \text{ on } \{t<\tau\}.
	\]
	Noticing that $\{T<\tau\}=\big\{\frac{\d\Q }{\d\P}>0\big\}$ we deduce that
	\[
	\frac{\d\Q }{\d\P} =\1{\frac{\d\Q }{\d\P}>0}\exp\Big( L_T-\frac12\langle L\rangle_T\Big).
	\]
	Now recall that $\P$ is, by definition, the unique weak solution of \eqref{eq.dyn.P} and thus, by \cite[Theorem III.4.29]{jacod2003limit} the predictable martingale representation property holds for $(\F,\P)$-local martingales.
	That is, there exists a unique $\alpha\in \H^{2}_{\rm loc}$, such that 
	\(
	L_t=L_0+\int_0^t \alpha_s \cdot  \d W_s.
	\)
	Thus,
	\begin{align}\label{eq.density.rep}
	 \frac{\d\Q }{\d\P} =\1{\frac{\d\Q }{\d\P}>0}  Z_T, \text{ where, }\d Z_t=Z_t\alpha_t\cdot  \d W_t,\,  Z_0=\E\Big[\frac{\d\Q }{\d\P}\Big|\Fc_0\Big] , \P\as
	\end{align}
	Furthermore, observe that for any bounded continuous function $f:\R^\xdim\longrightarrow \R$ we have
	\[
	\E^\P\Big[ f(X_0) \frac{\d\muin }{\d\nu_0}(X_0) \Big]=\E^{\nu_0}\Big[f \frac{\d\muin }{\d\nu_0}\Big]=\E^{\muin}[f]=\E^{\Q\circ X_0^{-1}}[f]=\E^\Q[f(X_0)]=\E^\P\Big[ f(X_0) \E^\P\Big[ \frac{\d\Q }{\d\P}\Big|\Fc_0 \Big]  \Big].
	\]
	That is,
	\[
	M_0=\E^\P\Big[ \frac{\d\Q }{\d\P}\Big|\Fc_0 \Big] =  \frac{\d\muin }{\d\nu_0}(X_0), \P\as 
	\]
	With this, we conclude that $\Q = \P^\alpha \in \Ac(\muin)$.\medskip
	
	Let us now note that by Itô's formula, on $\big\{\frac{\d\Q }{\d\P}>0\big\}$ we have that
	\[
	\ell(M_T)=\ell(M_0)+ \frac12   \int_0^T \ell^{\prime\prime}(Z_t)Z_t^2 \|\alpha_t\|^2 \d t+ \int
	_0^T\ell^\prime(Z_t) Z_t \alpha_t \cdot  \d W_t, \,\P\as,
	\]
	and by \cite[Proposition VIII.1.2]{revuz1999continuous}, $M$ is strictly positive $\Q\as$
	Thus, a localization argument leads to
\[
J(\P^\alpha)=\Ic_\ell(\muin|\nu_0)+ \frac12 \E^{\P^\alpha}\bigg[ \int_0^T \ell^{\prime\prime}(Z_t)Z_t \|\alpha_t\|^2 \d t\bigg]=\E^\Q\Big[\Big(\frac{\d\Q }{\d\P}\Big)^{-1} \ell\Big(\frac{\d\Q }{\d\P}\Big)\Big]=\E\Big[\ell\Big(\frac{\d\Q }{\d\P}\Big)\Big]=\Ic_\ell(\Q|\P)<\infty.
\]
The converse direction of the correspondence follows exchanging the roles of $\Ic_\ell(\Q|\P)$ and $J(\P^\alpha)$ in the above line.
\medskip

{\rm $(ii)$} The proof of the second part is now immediate since by {\rm $(i)$} we have that $\{\P^\alpha\in \Ac(\muin,\mufin): J(\P^\alpha)<\infty\} = \{\Q\in \Pc(\muin,\mufin): \Ic_\ell(\Q|\P)<\infty\}$.
\end{proof}
\end{lemma}

As a direct corollary, and in the case where the reference and admissible measures' initial condition coincide, we obtain the following result.
	It can be seen as an instance of the chain rule or tensorization of sorts for the divergence operator. 
	It is well-known in the case of the relative entropy, see, e.g., \citet*{lacker2018law} or \citet*[Theorem 2]{leonard2014some}.

\begin{corollary}\label{cor.chain}
	Let {\rm \Cref{assump.data}} hold and suppose $\muin=\nu_0$. 
	Then, for every $\Q \in \Pc(\muin)$ we have
	\begin{equation*}
		\Ic_\ell(\Q|\P) = \int_{\R^m}\Ic_\ell(\Q_x|\P_x)\muin(dx) .
	\end{equation*}
\end{corollary}
\begin{proof}
	We know that for every $\Q\in \Pc(\muin)$, it holds $\frac{\d\Q }{\d\P} =\1{\frac{\d\Q}{\d\P\hfill}>0}  Z_T,\, \d Z_t=Z_t \alpha_t\cdot \d W_t,\, \P\as$ and
	\begin{align*}
		\Ic_\ell(\Q|\P) &= \Ic_\ell(\muin|\nu_0)+ \frac12 \E^{\P^\alpha}\bigg[ \int_0^T \ell^{\prime\prime}(Z_t)Z_t \|\alpha_t\|^2 \d t\bigg] = \frac12  \int_{\R^m} \E^{\P^\alpha_x}\bigg[ \int_0^T \ell^{\prime\prime}(Z_t)Z_t \|\alpha_t\|^2 \d t\bigg]\muin(dx).
	\end{align*}
	Applying It\^o's formula and using $\ell(1) = 0$, we have $\frac12 \E^{\P^\alpha_x}\bigg[ \displaystyle \int_0^T \ell^{\prime\prime}(Z_t)Z_t \|\alpha_t\|^2 \d t\bigg] = \Ic_\ell(\P^\alpha_x|\P_x)$. 
\end{proof}

As corollary of the above result, we show that the functional $\Psi^\varphi$ introduced in \eqref{eq:defPsi} is actually the value function of an optimal stochastic control problem.
\begin{corollary}
\label{lem:Phi=oce}
	Let {\rm \Cref{assump.data}} hold.
	For every $\varphi\in C_{b,p}(\R^\xdim)$, the functional $\Psi^\varphi$ admits the representation
	\begin{align*}
		\Psi^\varphi(\mu) &= \inf_{\P^\alpha\in \Ac(\mu )} \E^{\P^\alpha} \bigg[ \frac12  \int_0^T  \ell^{\prime\prime}(Z_t)Z_t \|\alpha_t\|^2 \d t  +  Q_c \varphi(X_T) + C(X) \bigg]  +  \Ic_\ell(\mu|\nu_0)\\ 
		&=  \inf_{\P^\alpha\in \Ac(\mu)}\Big(\Ic_\ell\big(\P^{\alpha}|\P\big)   + \E^{\P^\alpha}[Q_c \varphi(X_T) +C(X)]\Big).
	\end{align*}
\end{corollary}
\begin{proof}
Let $\varphi\in C_{b,p}(\R^\xdim)$.
	Suppose $\mu(x:\E^{\P_x}[\ell^\ast((Q_c\varphi (X_T) +C(X))^+)]=\infty)>0$.
	Notice that since $\ell^\ast$ is increasing, we have that $\E[\ell^\ast( \zeta )]\leq \E[\ell^\ast(\zeta^+)]$ for any random variable $\zeta$.
	It then follows by definition that $\Psi^\varphi(\mu)=-\infty$.
	Moreover, thanks to \cite[Theorem 2.3]{cherny2007divergence}, the other two expressions in the statement are also equal to $-\infty$.
	\medskip
	
	We now consider the case $\mu(x:\E^{\P_x}[\ell^\ast((Q_c\varphi (X_T) +C(X))^+)]<\infty)=1$.
	By \Cref{prop:dual.f.div}$.(ii)$, we have that
	\begin{align}
		\Psi^\varphi(\mu)
		&=  -  \int_{\R^\xdim} \sup_{\Q\ll\P_x}\Big( \E^{\Q}[-Q_c\varphi (X_T) - C(X)] - \Ic_\ell(\Q|\P_x)\Big)  \mu(\d x) +  \Ic_\ell(\mu|\nu_0) \label{eq.OCE.rep.1} \\
		&=  \int_{\R^\xdim}   \inf_{\Q\ll\P_x}\Big(  \Ic_\ell(\Q|\P_x)+ \E^{\Q}[Q_c \varphi(X_T) - C(X)]\Big) \mu(\d x) +  \Ic_\ell(\mu|\nu_0)  . \notag
	\end{align}

Note that $\Q\ll\P_x$ implies $\Q\circ X_0^{-1}=\delta_x=\P_x\circ X_0^{-1}$, and recalling $\ell(1)=0$, we have that $\Ic_\ell(\Q\circ X_0^{-1}|\P_x\circ X_0^{-1})=0$.
	It them follows from \Cref{lemma.ref} that
	\begin{align*}
		\Psi^\varphi(\mu)&= \int_{\R^\xdim}  \bigg( \inf_{\P^\alpha\in \Ac(\delta_x)}\E^{\P^\alpha} \bigg[ \frac12  \int_0^T  \ell^{\prime\prime}(Z_t)Z_t \|\alpha_t\|^2 \d t  +  Q_c \varphi(X_T) + C(X)] \bigg]\bigg) 
		\mu (\d x) +  \Ic_\ell(\mu|\nu_0)\\
		&=\inf_{\P^\alpha\in \Ac(\mu )}\E^{\P^\alpha} \bigg[ \frac12  \int_0^T  \ell^{\prime\prime}(Z_t)Z_t \|\alpha_t\|^2 \d t  +  Q_c \varphi(X_T) +C(X)\bigg]  +  \Ic_\ell(\mu|\nu_0)\\
		&=\inf_{\P^\alpha\in \Ac(\mu )}\Big(\Ic_\ell\big(\P^{\alpha}|\P\big)   + \E^{\P^\alpha}[Q_c \varphi(X_T) + C(X)]\Big)
	\end{align*}
where the second equality follows from \Cref{lemma.v.disintegration}.
\end{proof}

\section{Proofs of the main results}
\label{sec:proofs_of_the_main_results}

This section collects the proofs of the main results as stated in \Cref{sub:main_results}.

\subsection{Existence, uniqueness and convex duality}


\begin{proof}[Proof of {\rm Theorem \ref{thm:exists.duality}}]
\begin{enumerate}[label=$(\roman*)$, ref=.$(\roman*)$,wide,  labelindent=0pt]
\item Since $\Pc_c(\muin,\mufin) \neq \emptyset$, for every $n\ge1$ there is $\Q^n\in \Pc_c(\muin,\mufin)$ such that
	\begin{equation*}
			V_c(\muin,\mufin) \ge \E^{\Q^n}[C(X)] + \Ic_\ell\big(\Q^{n} |\P\big) - \frac1n.
	\end{equation*}	
	As argued in the proof of \Cref{prop.v.lsc} above, there is $\Q^\star \in \Pc_c(\muin,\mufin)$ such that $\E^{\Q^\star}[C(X)] +\Ic_\ell(\Q^\star|\P) \le \liminf_{n\to\infty}(\E^{\Q^n}[C(X)] + \Ic_\ell(\Q^n |\P) )$,
	Thus, $\Q^\star$ is optimal for Problem \eqref{eq.SP}.\medskip
	
	Uniqueness follows since by \Cref{lemma.convexcontrols} the set $\Pc_c(\muin,\mufin)$ is convex, and under {\rm \Cref{assump.data}} the map \( \Q  \longmapsto \Ic_\ell(\Q|\P) \) is strictly convex.
	Consequently, $V_c(\muin,\mufin)$ admits a unique minimizer $\Q^\star\in \Pc_c(\muin,\mufin)$.

	\medskip

\item 
	For any $\muin,\mufin\in \Pc_p(\R^m)$ let us put
	\begin{equation*}
		D(\muin,\mufin) := \sup_{\varphi\in C_{b,p}(\R^m)}\Big(\Psi^\varphi(\muin) +  \langle \varphi, \mufin\rangle   \Big)
	\end{equation*}
	with $\Psi^\varphi$ defined in \eqref{eq:defPsi}.
	If $V_c(\muin, \mufin) = + \infty$ for all $\mufin \in \Pc_2(\R^m)$,
	then by \Cref{lem:Phi=oce},
	 \begin{align*}
		D(\muin,\mufin) & \ge  
		 \inf_{\P^\alpha \in \Ac(\muin)} \Ic_\ell\big(\P^{\alpha}|\P\big) + \langle  Q_c \varphi, \P^\alpha\circ X^{-1}_T\rangle +\E^{\P^\alpha}[C(X)]   - \langle \varphi, \mufin\rangle
	\end{align*}
	for all $\varphi\in C_{b,p}(\R^d)$.
	This shows, due to the convex dual representation of $\Wc_c(\P\circ X_T^{-1},\mufin)$ that $D(\muin, \mufin)\ge \inf_{\P^\alpha \in \Ac(\muin)} \Ic_\ell\big(\P^{\alpha}|\P\big) + \E^{\P^\alpha}[C(X)] + \Wc_c(\P^\alpha\circ X_T^{-1},\mufin)\ge  \infty$. 
	In particular, $V_c(\muin, \mufin) = D(\muin, \mufin)$.\medskip
	
	Let us now assume that there is $\mufin \in \Pc_2(\R^m)$ such that $V_c(\muin, \mufin) < + \infty$.
	Denote by $\Mc_p(\R^m)$ the space of Borel signed measures on $\R^d$ with finite second moment, and equipped with the weak topology.
	We extend the function $V_c(\muin,\cdot)$ to be $+\infty$ outside $\Pc_p(\R^m)$, and still denote by $V_c$ this extension.
	It is standard that the function $V_c(\muin,\cdot)$ remains convex and lower-semicontinuous on $\Mc_p(\R^m)$, recall \Cref{lemma.convexcontrols} and \Cref{prop.v.lsc}.
	Thus, by Fenchel-Moreau theorem, see \cite[Theorem 2.3.3]{zalinescu2002convex} we have
	\begin{align*}
		V_c(\muin,\mufin) = \sup_{\varphi \in C_{b,p}(\R^\xdim)}(\langle \varphi, \mufin \rangle - V^\ast_c(\varphi))= \sup_{\varphi \in C_{b,p}(\R^\xdim)}( - V^\ast_c(-\varphi)-\langle \varphi, \mufin \rangle),
	\end{align*}
	where the second equality follows by a monotone convergence argument, and $V^\ast_c$ is the convex conjugate of $V_c$ given by
	\begin{equation*}
		V^\ast_c(\varphi) \coloneqq \sup_{\mu \in \Mc_p(\R^\xdim)}(\langle \varphi, \mu \rangle - V_c(\muin,\mu)),\quad \varphi \in C_{b,p}(\R^\xdim).
	\end{equation*}

	The rest of the proof consists of showing that $-V_c^*(-\varphi) = \Psi^\varphi(\muin)$.
	To see that $-V_c^*(-\varphi) \geq \Psi^\varphi(\muin)$, note that
	\begin{align*}
		-V^*_c(-\varphi) &= \inf_{\mu \in \Pc_p(\R^\xdim)}(  V_c(\muin,\mu)+\langle \varphi, \mu\rangle )\\
			&=  \inf_{\mu \in \Pc_p(\R^\xdim)} \inf_{\P^\alpha\in \Ac_c(\muin,\mu)}\bigg (\frac12\E^{\P^\alpha}\bigg[ \int_0^T \ell^{\prime\prime}(Z_t)Z_t \|\alpha_t\|^2 \d t+ C(X) \bigg]   +\langle \varphi, \mu\rangle \bigg)+\Ic_\ell(\muin|\nu_0) \\
			&\geq  \inf_{\mu \in \Pc_p(\R^\xdim)} \inf_{\P^\alpha\in \Ac_c(\muin,\mu)}
				\bigg (\frac12\E^{\P^\alpha}\bigg[ \int_0^T \ell^{\prime\prime}(Z_t)Z_t \|\alpha_t\|^2 \d t +Q_c\varphi(X_T)+ C(X) \bigg]+\Ic_\ell(\muin|\nu_0)\\
			&= \inf_{\P^\alpha\in \Ac(\muin)}\frac12 \E^{\P^\alpha} \bigg[ \int_0^T \ell^{\prime\prime}(Z_t)Z_t \|\alpha_t\|^2 \d t +Q_c\varphi(X_T) + C(X)\bigg]+\Ic_\ell(\muin|\nu_0)\\
			&=\Psi^\varphi(\muin) 
	\end{align*}
where the second equality follows from \eqref{lemma.ref}, the inequality follows from \eqref{eq.dualitywot} since $\P^\alpha\in \Ac_c(\muin,\mufin)$, the third equality follows since $\Ac(\muin)=\bigcup_{\mu\in \Pc_p(\R^\xdim)} \Ac_c(\muin,\mu)$, and the last one is due to \Cref{lem:Phi=oce}.
The reversed inequality is argued as in \cite{hernandez2025schrodinger} following ideas from \cite[Lemma 2.1]{backhoff2019existence}.\qedhere
\end{enumerate}
\end{proof}

\subsection{Existence of optimal potentials}\label{sec.dual.optimizers}
In this section, we prove the existence of dual optimizers for Schrödinger’s nonentropic problem. 
	The proof relies on compactness arguments and a careful study of the dual functional. 
	This result provides a rigorous foundation for the dual formulation and will be used in subsequent results.
\begin{proof}[Proof of {\rm \Cref{thm.existence.dual}}]
	For each $n \in \N$, let $\Q^n\in \Pc_c(\muin, \mufin)$ and $\varphi^n\in C_{b,p}(\R^\xdim)$ be such that
	\begin{equation}
	\label{eq:1.noptim.potential}
		\E^{\Q^n}[C(X_0,X_T)] + \Ic_\ell(\Q^n|\P) - \frac1n \le V_c(\muin,\mufin) \le  \Psi^{\varphi^n}( \muin) - \langle \varphi^n, \mufin\rangle + \frac1n.
	\end{equation}
	Re-arranging terms and using definition of $\Psi^{\varphi^n}$, see \eqref{eq:defPsi}, this amounts to
	\begin{align}
	\notag
		& \int _{\R^\xdim} \Phi_{\P_x}\Big(-\varphi^n(X_T) -C(X_0,X_T) + \frac1n\Big)  \muin(\d x)\\ \notag
		&\le \langle -\varphi^n, \mufin\rangle - \E^{\Q^n}[C(X_0,X_T) ] + \frac1n -\Ic_\ell(\Q^n|\P) +\Ic_\ell(\muin|\nu_0)  \\ \notag
		&\leq \E^{\Q^n}\Big[-\varphi^n(X_T) - C(X_0,X_T) + \frac1n\Big] -  \int _{\R^\xdim} \Ic_\ell(\Q^n_x|\P_x) \muin(\d x)  + \E^{\Q^n}[-\varphi^n(X_T)] - \langle \varphi^n, \mufin\rangle \\ 
		\label{eq:approx.optim.potential}
		&\le  \int _{\R^\xdim} \bigg( \E^{\Q^n_x }\Big[-\varphi^n(X_T) - C(X_0,X_T) + \frac1n\Big] -  \Ic_\ell(\Q^n_x|\P_x)\bigg)  \muin(\d x) 
	\end{align}
	where we used the fact that $\Phi_{\P_x}(\xi + r) = \Phi_{\P_x}(\xi) -r $ for all $r\in \R$, the second inequality follows from \Cref{assump.data.2}, and the last inequality follows by the fact that $0 =\Wc_c(\Q^n\circ X_T^{-1}, \mufin) \ge \E^{\Q^n}[-\varphi^n(X_T)] - \langle \varphi^n, \mufin\rangle $.
	Next, using the convexity and growth property of $\ell$ as in \Cref{prop.v.lsc}, it follows that the sequence $(\frac{\d \Q^n_x }{\d\P_x})_{n\ge1}$ is uniformly integrable with respect to $\P_x$.
	Thus, by the de la Vall\'ee Poussin theorem, up to a subsequence, $(\frac{\d\Q^n_x}{\d\P_x})_{n\ge1}$ converge weakly in $\L^1(\P_x)$ to a positive random variable $Z_x$ with $\E^{\P_x}[Z_x] =1$.
	Thus, we can define a probability measure $\Q_x$ by $\Q_x(A)\coloneqq \E^{\P_x}[Z_x\cdot  \mathds{1}_A  ]$ and set $\Q\coloneqq \muin\otimes\Q_\cdot$.
	In addition, by Mazur's theorem, there is a subsequence in the asymptotic convex hull of $(\frac{\d\Q^n_x}{\d\P_x})_{n\ge1}$ which converges to $Z_x$ in $\L^1(\P_x)$ and therefore $\P_x\as$, up to a further subsequence.
	That is there are convex weights $\lambda^{i,j}\in [0,1], i,j\geq1$, such that $\sum_{i\ge j}\lambda^{i,j}=1$ and
	\[
		\frac{\d \widetilde{\Q}^n_x}{\d\P_x} := \sum_{j\ge n}\lambda^{n,j}\frac{\d\Q^j_x}{\d \P_x} \longrightarrow Z_x, \; \P_x\as,\; n\longrightarrow \infty.
	\]
	
To ease the notation, we put $\widetilde{\varphi}^n(x) \coloneqq \sum_{j\ge n}\lambda^{n,j}\varphi^j(x)$.
Let us now note that since $\Q^j\circ X^{-1}_T = \mufin$ for all $j$ and thus $\widetilde{\Q}^n\circ X_T^{-1} = \mu_T$, we have that
\begin{align}\label{eq.optim.potential.3}
\notag \sum_{j\ge n}\lambda^{n,j} \bigg(  \int _{\R^\xdim}    \E^{\Q^j_x } [-\varphi^j(X_T) ] \muin(\d x)  \bigg) 
&=
\sum_{j\ge n}\lambda^{n,j}   \E^{\Q^j } [-\varphi^j(X_T) ]  \\
\notag &=
\sum_{j\ge n}\lambda^{n,j}  \int _{\R^\xdim}    -\varphi^j(x)   \mufin(\d x)\\
&=
 \int _{\R^\xdim}  - \widetilde{\varphi}^n (x)    \mufin(\d x)
=
\int _{\R^\xdim} \Big(  \E^{\widetilde{\Q}^n_x }\big[ - \widetilde{\varphi}^n (X_T) \big]\Big) \muin(\d x) .
\end{align}
Consequently, by convexity of $\Phi_{\P_x}(\cdot)$ and $\Ic_\ell(\cdot|\P_x)$, it follows from \eqref{eq:approx.optim.potential} and \eqref{eq.optim.potential.3} that 
	\begin{align}\label{eq.optim.potential.4}
		 &\int _{\R^\xdim} \Phi_{\P_x}  \Big(-\widetilde{\varphi}^n (X_T) - C(X_0,X_T)+ \frac1n \Big)  \muin(\d x)\notag  \\
		&\le \sum_{j\ge n}\lambda^{n,j}   \int _{\R^\xdim} \Phi_{\P_x}  \Big( - \varphi^j(X_T)- C(X_0,X_T) + \frac1n\Big)  \muin(\d x) \notag  \\
		&\le \sum_{j\ge n}\lambda^{n,j}  \int _{\R^\xdim} \bigg(   \E^{\Q^j_x }\Big[-\varphi^j(X_T) - C(X_0,X_T) + \frac1n\Big] - \Ic_\ell(\Q^j_x|\P_x) \bigg ) \muin(\d x)\notag   \\ 
		&\le   \int _{\R^\xdim} \bigg(  \E^{\widetilde{\Q}^n_x }\Big[ - \widetilde{\varphi}^n ( X_T) - C(X_0,X_T) + \frac1n\Big] - \Ic_\ell(\widetilde{\Q}^n_x|\P_x) \bigg ) \muin(\d x). 
\end{align}

	We now claim that $\E^{\mu_0\otimes \P_\cdot}[\ell^\ast((-\widetilde{\varphi}^n(X_T) - C(X_0,X_T) + \frac1n )^+)]<+\infty$.
	Indeed, if not, as in the proof of \Cref{lem:Phi=oce}, we deduce that $\Psi^{\tilde \varphi^n}(\mu_0)=-\infty$ which in light of \eqref{eq:1.noptim.potential} leads to a contradiction.
	Thus, as in the proof of \Cref{thm:propertyMain}, we may use the convex dual representation of $\Phi_{\P_x}$, \Cref{prop:dual.f.div}, together with {\color{black} $\P_{0T}\sim\muin\otimes\mufin$}, to deduce that $\widetilde{\Q}^n_x$ attains the supremum in \eqref{eq:dual.rep.gen.oce} for $\Phi_{\P_x}(-\widetilde{\varphi}^n(X_T) - C(X_0,X_T) + \frac1n )$ and thus by \Cref{prop:dual.f.div}, it holds
	\begin{align*}
		\frac{\d\widetilde{\Q}^n}{\d\P} =\dfrac{\d \muin}{\d \nu_0}(X_0) \partial_x\ell^*\Big( -\widetilde{\varphi}^n(X_T) - C(X_0,X_T) + \frac1n - \widetilde\psi^n(X_0) \Big),\,\, \P\as,
	\end{align*}
	for a Borel-measurable function $\widetilde\psi^n$.
	Because $\partial_x\ell^*$ is invertible, indeed $\partial_x\ell^*= (\partial_x\ell)^{-1}$, we have\footnote{Recall $\P_{0T}\sim\muin\otimes\mufin$ implies $\muin\sim\nu_0$.}
	\begin{align}\label{eq.optim.potential.2}
		-\widetilde{\varphi}^n(X_T) - C(X_0,X_T)+ \frac1n - \widetilde\psi^n(X_0) = \partial_x\ell \Big( \dfrac{\d\widetilde{\Q}^n}{\d\P}/\dfrac{\d \muin}{\d \nu_0}(X_0) \Big ).
	\end{align}
 %
 %
 Since $\ell$ is continuosly differentiable, there exists a Borel-measurable function $f:\R^\xdim\times\R^\xdim\to \R$ such that $\widetilde{\varphi}^n(x)+\widetilde{\psi}^n(y) \longrightarrow -\partial_x\ell (f(x,y)) -C(x,y), \P_{0T}\ae$
	Moreover, the convergence of $(\frac{\d\tilde\Q^n_x}{\d\P_x})_{n\ge1}$ in $\L^1(\P_x)$, \eqref{eq.optim.potential.2}, and the estimate $\partial_x\ell(x) \le \varepsilon \ell(x)+C, \eps>0$, see \cite[Theorem 6.7]{evans2015measure}, shows that the convergence holds in $\L^1(\P_{0T})$.
	It then follows from \cite[Proposition 2]{ruschendorf1993note} (recall $\P_{0T}\sim\muin\otimes\mufin$) that there are two Borel functions $\varphi^\star$ and $\psi^\star$ such that $ -(\partial_x\ell^*)^{-1}(f(x,y)) -C(x,y) = \varphi^\star(x) + \psi^\star(y)$, $\P_{0T}\ae$
	That is,
\begin{align*}
	\frac{\d\Q }{\d\P} =\dfrac{\d \muin}{\d \nu_0}(X_0) \partial_x\ell^*\Big( - \varphi^\star(X_T) - C(X_0,X_T)   - \psi^\star(X_0) \Big), \P\as,
\end{align*}	
Notice this implies that $\E^{\muin\otimes  \P_\cdot }[\partial_x\ell^\ast (  - \varphi^\star(X_T) - C(X_0,X_T)   - \psi^\star(X_0) )] = 1$ and, by \Cref{prop:dual.f.div}, we have that
\begin{align*}
-\int_{\R^\xdim}  \Phi_{\P_x }\Big( - \varphi^\star(X_T) - C(X_0,X_T) \Big)\muin(\d x) = - \E^{\mu_0\otimes\P_\cdot}\big[ \ell^\ast( -\varphi^\star(X_T) - C(X_0,X_T)-\psi^\star(X_0))] - \langle \psi^\star,\mu_0\rangle .
\end{align*}

Back to \eqref{eq:1.noptim.potential}, using again convexity of $\Phi_{\P}(\cdot)$ and \Cref{prop:dual.f.div}, we have
\begin{align*}
	V_c(\muin,\mufin) &\le \sum_{j\ge n}\lambda^{n,j}\Big( - \langle \Phi_{\P_\cdot}(-\varphi^n(X_T) - C(X_0,X_T)), \muin \rangle - \langle \varphi^n, \mufin\rangle + \frac1n  \Big)\\
		&\le - \int_{\R^\xdim}  \Phi_{\P_x } \Big( - \widetilde{\varphi}^n(X_T) - C(X_0,X_T)   \Big) \muin( \d x)   - \langle \widetilde{\varphi}^n, \mufin\rangle   + \frac1n  \\ 
		&=- \int_{\R^\xdim}   \E^{\P_x }\Big[ \ell^\ast\Big ( -\widetilde \varphi^n(X_T) - C(X_0,X_T) -\widetilde \psi^n(X_0)\Big) + \widetilde \psi^n (X_0) \Big] \muin (\d x)  - \langle \widetilde \varphi^n, \mufin\rangle   + \frac1n.
\end{align*}
We now note that by definition of $\ell^\ast$, recall $\ell(1)=0$, we have that $-\ell^\ast(-x)\leq x$.
	Thus, since $\P_{0T}\sim\muin\otimes\mufin$ and $\widetilde{\varphi}^n(x)+\widetilde{\psi}^n(y)\longrightarrow  \varphi^\star(x) + \psi^\star(y)$ in $\L^1(\P_{0T})$, taking the limit 
\begin{align*}
	V_c(\muin,\mufin) &\leq  - \E^{\mu_0\otimes\P_\cdot}\big[ \ell^\ast( -\varphi^\star(X_T) - C(X_0,X_T)-\psi^\star(X_0))] - \langle \psi^\star,\muin\rangle - \langle \varphi^\star,\mufin\rangle\\
	&= - \int_{\R^\xdim}  \Phi_{\P_x }\Big( -  \varphi^\star (X_T) - C(X_0,X_T) \Big)\muin(\d x)  - \langle \varphi^\star,\mufin\rangle = \Phi^{\varphi^\star}(\muin)- \langle \varphi^\star,\mufin\rangle.\qedhere 
\end{align*}
%
\end{proof}

\begin{remark}
Let us comment that, in the case of a general weak cost terminal constraint, the use of {\rm \cite[Proposition 2]{ruschendorf1993note}} guarantees the summable structure of the limit, but not necessarily that the limit would be of the form $Q_c\varphi^\star+\psi^\star$ for some functions $\varphi^\star,\psi^\star$.
\end{remark}

\subsection{Dual characterization of primal optimizers}


\begin{proof}[Proof of {\rm \Cref{thm:propertyMain}}]
Let us assume that {\rm \Cref{assump.data.2}}$.(ii)$ holds, i.e., $\Ic_\ell$ being superadditive. 
	Since $\varphi^\star$ attains the maximum in the dual problem and $\Q^\star$ attains the minimum in the primal problem, it holds
	\begin{align}\label{eq.thm.i.aux.1}
	 \E^{\Q^\star}[C(X)] +	\Ic_\ell(\Q^\star |\P) = V_c(\muin,\mufin) =  \Psi^{\varphi^\star} (\muin )  - \langle  \varphi^\star , \mufin\rangle.
	\end{align}

We claim that $\langle  \varphi^\star , \mufin\rangle= \E^{\Q^\star}[Q_c\varphi^\star(X_T)]$.
	Indeed, since $\Q^\star$ is feasible for the primal problem we have that $\Wc_c(\Q^\star\circ X_T^{-1},\mu_T)=0$ and thus it follows from \eqref{eq.dualitywot} that $ \E^{\Q^\star}[Q_c\varphi^\star(X_T)]\leq \langle  \varphi^\star , \mufin\rangle$.
	Suppose that $ \E^{\Q^\star}[Q_c\varphi^\star(X_T)]-\langle  \varphi^\star , \mufin\rangle<-\eps$, for some $\eps>0$.
	Since $\Q^\star\in \Ac(\muin)$, thanks to \Cref{lem:Phi=oce} and \eqref{eq.thm.i.aux.1} we have  
\begin{align*}
	\E^{\Q^\star}[C(X)] + \Ic_\ell(\Q^\star |\P) &=\inf_{\P^\alpha\in \Ac(\muin)}\Big(\Ic_\ell\big(\P^{\alpha}|\P\big)   + \E^{\P^\alpha}[Q_c \varphi^\star(X_T) + C(X)]\Big) - \langle  \varphi^\star , \mufin\rangle\\
&\leq  \Ic_\ell\big(\Q^\star |\P\big)   + \E^{\Q^\star}[Q_c \varphi^\star(X_T) +C(X)] - \langle  \varphi^\star , \mufin\rangle\\
&< \Ic_\ell\big(\Q^\star |\P\big) + \E^{\Q^\star}[C(X)] -\eps,
\end{align*}
which is a contradiction.
Thus, $\langle  \varphi^\star , \mufin\rangle= \E^{\Q^\star}[Q_c\varphi^\star(X_T)]$.\medskip

Consequently, \eqref{eq.thm.i.aux.1} and \eqref{eq.OCE.rep.1} imply that 
\begin{align}\label{eq.thm.i.aux.2}\begin{split}
	  \Ic_\ell(\Q^\star |\P)&  +  \E^{\Q^\star}[Q_c\varphi^\star(X_T) + C(X)] =\Psi^{\varphi^\star} (\muin )\\
	  &=   -  \int_{\R^\xdim} \sup_{\Q\ll\P_x}\Big( \E^{\Q}[-Q_c\varphi^\star (X_T) - C(X)] - \Ic_\ell(\Q|\P_x)\Big)  \muin(\d x) +  \Ic_\ell(\muin|\nu_0) .
	  \end{split}
\end{align}
Moreover, notice that because $\Ic_\ell$ is superadditive relative to $\P$, we also have
\begin{align} 
- \Ic_\ell(\Q^\star|\P) - \E^{\Q^\star}[Q_c\varphi^\star(X_T) + C(X)]  & \leq  \int_{\R^m} \Big( \E^{\Q^\star_x}[-Q_c\varphi^\star(X_T) - C(X)] -\Ic_\ell(\Q^\star_x|\P_x)\Big) \muin(dx) - \Ic_\ell(\muin|\nu_0) \label{eq.thm.i.aux.3} \\
&\leq \int_{\R^\xdim} \sup_{\Q\ll\P_x}\Big( \E^{\Q}[-Q_c\varphi^\star (X_T) - C(X)] - \Ic_\ell(\Q|\P_x)\Big)  \muin(\d x)- \Ic_\ell(\muin|\nu_0), \notag
\end{align}
which together with \eqref{eq.thm.i.aux.2} implies that

\[
 \int_{\R^\xdim} \sup_{\Q\ll\P_x}\Big( \E^{\Q}[-Q_c\varphi^\star(X_T) - C(X)] - \Ic_\ell(\Q|\P_x)\Big)  \muin(\d x)= \int_{\R^\xdim} \Big(   \E^{\Q^\star_x}[- Q_c\varphi^\star(X_T) -  C(X)]- \Ic_\ell(\Q^\star_x |\P_x )  \Big)  \muin(\d x).
\]

It then follows from \Cref{prop:dual.f.div} that
	\begin{align}\label{eq.thm.i.aux.4}
		\frac{\d \Q^\star_x}{\d \P_x} = \partial_x\ell^*\big(-Q_c\varphi^\star(X_T) - C(X) - \psi^\star_x \big), {\color{black} \P_x\as}, \text{for }\mu_0\ae\, x\in \R^\xdim,
	\end{align}
where $\psi^\star_x \in \R$ denotes the unique real number satisfying $\E^{ \P_x}[\partial_x\ell^*(-Q_c\varphi^\star(X_T)- C(X) - \psi^\star_x)] = 1$.\medskip

	We now show that the map $\iota:\R^\xdim\ni x\longmapsto \psi^\star_x$ is $\Bc(\R^\xdim)$--measurable. 
	To do so, recall that because $(\P_x)_{x\in \R^\xdim}$ denotes the r.c.p.d.~of $\P$ given $\Fc_0=\Bc(\R^\xdim)$, it holds that for any $B\in \Fc_T$, the function $\R^\xdim\ni x\longmapsto \E^{\P_x}[\1{B}]$ is $\Bc(\R^\xdim)$--measurable, and that, under \Cref{assump.data}, $ \ell^\star$ is continuously differentiable.
	We thus consider the measurable map $\pi:\R^\xdim\times\R\ni (x,\psi)\longmapsto  \E^{ \P_x}[\partial_x\ell^*(-Q_c\varphi^\star(X_T)- C(X) - \psi)]\in \R$.
	It remains to note that for any measurable set $B\in \Bc(\R^\xdim)$, $\pi^{-1}(\cdot,B)(1)=\iota^{-1}(B)$ is measurable.
\medskip

Thus, we may write \eqref{eq.thm.i.aux.4} as
	\begin{align}
	\label{eq:ref.ref.}
		\frac{\d \Q^\star_x}{\d \P_x} = \partial_x\ell^*(-Q_c\varphi^\star(X_T) - C(X)- \psi^\star(x)), {\color{black} \P_x\as}, \text{for }\mu_0\ae\, x\in \R^\xdim
	\end{align}
	for a measurable function $\psi:\R^\xdim \to \R$.
Let $A\in \Fc_T$ and note $A=\bigcup_{x\in \R^\xdim} A_x$, where $A_x\coloneqq\{ \omega \in A: X_0(\omega)=x\}$.
	Then, since $\P_{0T}\sim \muin\otimes\mufin$, by desintegration and \eqref{eq:ref.ref.} we have
\begin{align*}
\Q^\star(A)&=\int_{\R^\xdim} \int_{A_x} \Q^\star_x(\d \omega)  \muin(\d x)\\
		&=\int_{\R^\xdim} \dfrac{\d \muin}{\d \nu_0}(x)  \int_{A_x} \Q^\star_x(\d \omega)  \nu_0(\d x)\\
		&=\int_{\R^\xdim} \dfrac{\d \muin}{\d \nu_0}(x)  \int_{A_x}\partial_x\ell^*\big( -Q_c\varphi^\star(X_T) - C(X) - \psi^\star(X_0)\big)(\omega) \P_x(\d \omega)   \nu_0(\d x)\\
		&=\int_{A} \dfrac{\d \muin}{\d \nu_0}(X_0)  \partial_x\ell^*\big(-Q_c\varphi^\star(X_T)  - C(X)- \psi^\star(X_0)\big)  \P (\d \omega).
\end{align*}
This establishes the result.\medskip

Alternatively, if {\rm \Cref{assump.data.2}}$.(i)$ holds, i.e., $\muin=\nu_0$, then \eqref{eq.thm.i.aux.3} now follows thanks to \Cref{cor.chain}. This is the only step where superadditivity of $\Ic_\ell$ was needed.
Thus, the result follows as above.
\end{proof}

\begin{proof}[Proof of {\rm \Cref{thm:property.main.2}}]
We start by showing $(i)$.
Recall that, see {\rm \Cref{lemma.tv.ctransform}}, for $c(x,\rho)=\int_{\R^\xdim} \1{x\neq y}\rho(\d y)$, it holds $Q_c\hat \varphi(x)=\hat\varphi(x)$.
\medskip

First, we show that
	\begin{equation*}
		\frac{\d  \Q_x}{\d \P_x} = \partial_x\ell^*(-\hat \varphi(X_T) - C(X) - \hat \psi(x)),  {\color{black} \P_x\as}, \text{for }\mu_0\ae\, x\in \R^\xdim.
	\end{equation*}
Let $f\in \L^0(\Q)$.
By definition of $\Q$, we have
\begin{align*}
\int_{\Omega} f(\omega) \Q(\d\omega) &=\int_{\Omega} \dfrac{\d \muin}{\d \nu_0}(X_0)  \partial_x\ell^*(-\hat \varphi (X_T) -C(X) - \hat \psi (X_0)) f(\omega) \P (\d \omega)  \\
&=\int_{\R^\xdim} \dfrac{\d \muin}{\d \nu_0}(x)   \int_{\Omega_x}\partial_x\ell^*(-\hat \varphi(X_T)  -C(X)- \hat \psi (X_0)) f(\omega) \P_x(\d \omega)      \nu_0(\d x) \\
&=\int_{\R^\xdim} \int_{\Omega_x}\partial_x\ell^*(-\hat \varphi(X_T) - C(X) - \hat \psi(x))  f(\omega)   \P_x(\d \omega)   \muin  (\d x),
\end{align*}
but also,
\begin{align*}
\int_{\Omega} f(\omega) \Q(\d\omega) =\int_{\R^\xdim} \int_{\Omega_x} f(\omega)\Q_x(\d \omega)  \muin(\d x).
\end{align*}
The claim follows because by Radon-Nikodym theorem we have 
\begin{align*}
 \int_{\Omega_x}\partial_x\ell^*(-\hat \varphi(X_T) - C(X) - \hat \psi (x))  f(\omega)   \P_x(\d \omega)= \int_{\Omega_x} f(\omega)\Q_x(\d \omega),  \text{ for }\mu_0\ae\, x\in \R^\xdim.
 \end{align*}
 
 Second, thanks to {\rm \Cref{thm:exists.duality}}, we have that
 \begin{align*}
V_c(\muin,\mufin)&\geq 	  -\int_{\R^\xdim} \Phi_{\P_x}\big(-\hat \varphi(X_T) - C(X)\big) \muin(\d x) +\Ic_\ell(\muin|\nu_0)-\langle \hat \varphi, \mufin\rangle\\
& = \E^{ \Q}[\hat\varphi(X_T)+  C(X)]-\langle \hat \varphi, \mufin\rangle +\int_{\R^\xdim} \Ic_\ell( \Q_x|\P_x) \muin(\d x) +\Ic_\ell(\muin|\nu_0)\\
& = \int_{\R^\xdim} \Ic_\ell( \Q_x|\P_x) \muin(\d x) +\Ic_\ell(\muin|\nu_0) \geq \E^{ \Q}[  C(X)] + \Ic_\ell( \Q|\P) \geq V_c(\muin,\mufin)
\end{align*}
where the first equality follows from the first step and \eqref{prop:dual.f.div}.
	The second equality follows since $\Q\circ X_T^{-1}=\mufin$, as $\hat\Q\in \Pc_c(\muin,\mufin) $ for the choice of $c$ in the statement, and thus $\E^{\Q}[\hat\varphi(X_T)]=\langle \hat \varphi, \mufin\rangle$.
	The third step follows from \Cref{cor.chain} in the case $\muin=\nu_0$ or by the subadditivity of $\Ic_\ell$.
	 The last step follows by the definition of $V_c(\muin,\mufin)$.
	 The equality to $V_c(\muin,\mufin)$ of the second term implies that $\hat\varphi$ is dual optimal and equality to $\E^{ \Q}[  C(X)] + \Ic_\ell( \Q|\P)$ implies $\Q$ is primal optimal.

\medskip

We now argue $(ii)$.	
To this end, we first show that for $\Q$ as in the statement, we have that
	\begin{align}\label{eq.system.aux0}
			\dfrac{\d \overleftarrow \Q}{\d \overleftarrow \P} =\frac{\d\overleftarrow \Q_0 }{\d\overleftarrow \P_0} (\overleftarrow X_{\!0})   \partial_x\ell^*\big(  - \hat \varphi(\overleftarrow X_{\!T}) - C(\overleftarrow X) -\hat \psi(\overleftarrow X_{\!0})\big).
	\end{align}
First, letting $\{(\overleftarrow \Q)_x\}_{x\in \R^\xdim}$ denote the r.c.p.d.~of $\overleftarrow \Q$ with respect to $\sigma(\overleftarrow X_0)$, we claim that
\begin{align}\label{eq.system.aux1}
(\overleftarrow \Q)_x=\Q_x\circ \overleftarrow{\Tc}.
\end{align}
On the one hand, note that by definition of $(\overleftarrow \Q)_x$ we have that for any measurable $\xi:\Omega\longrightarrow \R$ and $f:\R\longrightarrow [0,\infty)$
 \[
\E^{ \overleftarrow\Q}[ \E^{{\overleftarrow \Q}_{\overleftarrow X_{\!0}}} [\xi]f( \overleftarrow X_{\!0}) ]= \E^{\overleftarrow \Q}[ \E^{\overleftarrow \Q}[\xi | \sigma(\overleftarrow X_{\!0})]f(\overleftarrow X_{\!0})]=\E^{\overleftarrow \Q}[  \xi f(\overleftarrow X_{\!0})]=\E^{ \Q}[  \xi\circ \overleftarrow{\Tc}^{-1} f( X_0)].
 \]
On the other hand, since $\E^{\Q_x\circ \overleftarrow{\Tc}}[\xi]=\E^{\Q_x}[\xi\circ \overleftarrow{\Tc}^{-1}]$, we have that
\[
\E^{ \Q}[ \E^{\Q_{X_0}\circ \overleftarrow \Tc}[\xi]f( X_0) ]=\E^{ \Q}[ \E^{\Q_{X_0}}[\xi\circ \overleftarrow{\Tc}^{-1}]f( X_0) ]=\E^{ \Q}[  \xi\circ \overleftarrow{\Tc}^{-1} f( X_0)].
\]

That is
\[
\E^{ \overleftarrow\Q}[ \E^{{\overleftarrow \Q}_{\overleftarrow X_{\!0}}} [\xi]f( \overleftarrow X_{\!0}) ] =\E^{ \Q}[ \E^{\Q_{X_0}\circ \overleftarrow \Tc}[\xi]f( X_0)].
\]
The claim follows from the arbitrariness of $\xi$ and $f$.

\medskip

We now establish \eqref{eq.system.aux0}. Notice that, as in the proof of {\rm \Cref{thm:propertyMain}}, we may deduce that
\begin{equation}\label{eq.system.aux2}
	\frac{\d \Q_x}{\d \P_x} = \partial_x\ell^*(-\hat \varphi(X_T) - C(X) - \hat \psi(x)),  {\color{black} \P_x\as}, \text{for }\mu_0\ae\, x\in \R^\xdim.
\end{equation}
Let $A\subseteq \Omega$, and note that
\begin{align*}
\overleftarrow \Q(A)=\int_\Omega \mathds{1}_A(\omega) \overleftarrow \Q(\d \omega)& = \int_{\R^\xdim}\int_\Omega \mathds{1}_A(\omega)\,    (\overleftarrow \Q)_x (\d \omega)  \overleftarrow \Q_{\!0} (\d x)\\
& =\int_{\R^\xdim} \int_\Omega \mathds{1}_A(\omega)\,    \Q_x\circ \overleftarrow \Tc (\d \omega)  \overleftarrow \Q_{\!0} (\d x)\\
& =\int_{\R^\xdim}\int_\Omega \mathds{1}_A(\overleftarrow \Tc^{-1} (\omega) )\, \partial_x\ell^*(-\hat \varphi(X_T) - C(X) - \hat \psi(X_0))  \P_x (\d \omega)  \overleftarrow \Q_{\!0} (\d x)\\
& =\int_{\R^\xdim}\frac{\d\overleftarrow \Q_0 }{\d\overleftarrow \P_0} (x)  \int_\Omega \mathds{1}_A(\omega) \partial_x\ell^*(-\hat \varphi(X_T\circ \overleftarrow \Tc) - C(\overleftarrow  X ) - \hat \psi(X_0\circ \overleftarrow \Tc))   \P_x\circ \overleftarrow \Tc (\d \omega)    \overleftarrow \P_{\!0} (\d x)\\
& =\int_{\R^\xdim}\int_\Omega \mathds{1}_A(\omega)\frac{\d\overleftarrow \Q_0 }{\d\overleftarrow \P_0} (x)  \partial_x\ell^*(-\hat \varphi( \overleftarrow  X_{\!T}) - C(\overleftarrow  X ) - \hat \psi( \overleftarrow  X_{\!0}))  ( \overleftarrow  \P)_x  (\d \omega)    \overleftarrow \P_{\!0} (\d x)\\
& = \int_\Omega \mathds{1}_A(\omega)\frac{\d\overleftarrow \Q_0 }{\d\overleftarrow \P_0} ( \overleftarrow  X_{\!0})  \partial_x\ell^*(-\hat \varphi( \overleftarrow  X_{\!T}) - C(\overleftarrow  X ) - \hat \psi( \overleftarrow  X_{\!0}))   \overleftarrow  \P  (\d \omega)
\end{align*}
where the third and sixth steps follow from \eqref{eq.system.aux0}, the fifth step is due to \eqref{eq.system.aux1}.\medskip

%
%
%
%

We now argue \eqref{eq.prop.system}. The first equation follows since $\Q_0=\muin$ using the density of $\Q$ with respect to $\P$. 
	Let us argue the second equation in \eqref{eq.prop.system}.
	Let $f:\R^\xdim\longrightarrow \R$ be bounded continuous and recall that
%
because $\Q_T=\mufin$, it follows from \eqref{eq.system.aux0} that
	\begin{align*}
			\int_{\R} f (x)\mufin(\d x)=\E^\Q[ f(X_T)] =\E^{\overleftarrow \Q}[ f( X_0)] &=\E^{\overleftarrow \P}\Big[ f( X_0) \dfrac{\d \overleftarrow\Q}{\d \overleftarrow \P} \Big]\\
			& = \int_{\R} f(x) \frac{\d\overleftarrow \Q_0 }{\d\overleftarrow \P_0} (x) \int_{\Omega}     \partial_x\ell^*\big(  - \hat \varphi(\overleftarrow X_{\!T}) - C(\overleftarrow X) -\hat \psi(\overleftarrow X_{\!0})\big)( \overleftarrow  \P)_x (\d \omega)\overleftarrow \P_{\!0}(\d x)\\
			& = \int_{\R} f(x)   \int_{\Omega}   \partial_x\ell^*\big(  - \hat \varphi(\overleftarrow X_{\!T}) - C(\overleftarrow X ) -\hat \psi(\overleftarrow X_{\!0})\big)( \overleftarrow  \P)_x (\d \omega)\mufin(\d x)\\
			& = \int_{\R} f(x)  \E^{\overleftarrow \P} \big[  \partial_x\ell^*\big(  - \hat \varphi(\overleftarrow X_{\!T}) - C(\overleftarrow X) -\hat \psi(\overleftarrow X_{\!0})\big) |  \overleftarrow X_{\!0}=x\big]\mufin(\d x).
	\end{align*}
	This ends the proof.
\end{proof}


\begin{proof}[Proof of {\rm \Cref{rem:entropic-Schroedinger}}]
If, in addition, $\P$ is reversible, we claim that $\E^\P[\xi|\sigma(X_0)]=\E^\P[\xi|\sigma(X_T)]$ for any measurable function $\xi$. Assuming the claim for a moment, we have that
\begin{align*}
&\int_{\R} f(x)  \E^{\overleftarrow \P} \big[  \partial_x\ell^*\big(  - \hat \varphi(\overleftarrow X_{\!T}) - C(\overleftarrow X) -\hat \psi(\overleftarrow X_{\!0})\big) |  \overleftarrow X_{\!0}=x\big]\mufin(\d x)\\
&=\int_{\R} f(x)  \E^{\overleftarrow \P_x \circ \overleftarrow \Tc } \big[  \partial_x\ell^*\big(  - \hat \varphi( X_T) - C( X) -\hat \psi( X_0)\big) \big]\mufin(\d x)\\
&=\int_{\R} f(x)  \E^{ \P_x  } \big[  \partial_x\ell^*\big(  - \hat \varphi( X_T) - C( X) -\hat \psi( X_0)\big) \big]\mufin(\d x)\\
&=\int_{\R} f(x)  \E^{ \P} \big[  \partial_x\ell^*\big(  - \hat \varphi( X_T) - C( X) -\hat \psi( X_0)|X_T=x]\big) \big]\mufin(\d x).
\end{align*}

It remains to show the $\E^\P[\xi|\sigma(X_0)]=\E^\P[\xi|\sigma(X_T)]$ for any measurable function $\xi$. 
This is obtained by 
\begin{align*}
\E^\P[\E^\P[ \xi |\sigma(X_0)]f(X_0)] &=\E^\P[  \xi f(X_0)]=\E^{\overleftarrow \P}[  \xi f(X_0)]=\E^{ \P}[  \xi\circ \overleftarrow \Tc^{-1}  f(X_T)]=\E^\P[\E^\P[ \xi\circ\overleftarrow \Tc^{-1}  |\sigma(X_T)]f(X_T)]\\
&=\E^\P[\E^{ \P_{X_T}\circ \overleftarrow \Tc } [ \xi ]f(X_T)] =\E^\P[\E^{ (\overleftarrow \P)_{X_T} } [ \xi ]f(X_T)] =\E^\P[\E^{ \P_{X_T} } [ \xi ]f(X_T)]\\
&=\E^\P[\E^{ \P} [ \xi  |\sigma(X_T)]f(X_T)]
\end{align*}
where the second and last steps follow from the reversibility of $\P$ and the sixth step is due to \eqref{eq.system.aux1} since $\sigma(X_T)=\sigma(\overleftarrow X_{\!0})$.
This concludes the proof.
\end{proof}

\subsection{Dual characterization of the marginals}

The proof of Theorem \ref{thm:propertyMain} will build upon elements of optimal stochastic control as well as properties of time-reversed diffusions.
Before presenting the proof, let us derive some intermediate results from stochastic optimal control theory.
In the rest of the paper, we assume $C=0$ and $b$ is Markovian and of gradient form, i.e., $b(t,x) = -\partial_x U(x(t))/2$, $(t,x)\in [0,T]\times C([0,T],\R^\xdim)$, for some continuously differentiable function $U:\R^\xdim\longrightarrow \R$.



\subsubsection{Tidbits of stochastic optimal control theory}


In this subsection we take a closer look at the stochastic  optimal control problem with value
\begin{equation}
\label{eq:cont.problem}
	\Psi^\varphi(\mu) - \Ic_\ell(\mu|\nu_0) = \inf_{\P^\alpha\in \Ac(\mu )}\E^{\P^\alpha} \bigg[ \int_0^T \frac12 \ell^{\prime\prime}(Z_t)Z_t \|\alpha_t\|^2 \d t  +  Q_c \varphi(X_T) \bigg],\; \d Z_t = Z_t\alpha_t\cdot \d W_t,\; {\color{black} Z_0=\dfrac{\d \mu}{\d \nu_0}(X_0)}.
\end{equation}
We have already seen that for every $\varphi\in C_{b,p}(\R^\xdim)$,  
we have
	\begin{align*}
		\Psi^\varphi(\mu) - \Ic_\ell(\mu|\nu_0) &= \inf_{\P^\alpha\in \Ac(\mu)}\Big(\Ic_\ell\big(\P^{\alpha}|\P\big)   + \E^{\P^\alpha}[Q_c \varphi(X_T)]\Big),
	\end{align*}
	see \Cref{lem:Phi=oce}.
	Since $\varphi$ (and thus $Q_c\varphi(X_T)$) is bounded from below, it follows exactly as in the proof of \Cref{prop.v.lsc} that an optimal measure $ \Q \in \Pc(\mu)$ exists.
	Thus, by \Cref{lemma.ref} we deduce that the problem \eqref{eq:cont.problem} admits an optimal $\P^{ \alpha^\star}\in \Ac(\mu)$ for a control process $\alpha^\star \in \H_{\text{loc}}^2$.
	Our goal in this section is to further analyze properties of the optimizer $\alpha^\star$.

\begin{proposition}
\label{prop:optimal.control.Markovian.entropy}
	Let {\rm \Cref{assump.data}} be satisfied and put $\ell(x) =  x\log(x) -x +1$.
	Then for any $\varphi\in B_{b,p}(\R^\xdim)$ the problem \eqref{eq:cont.problem} admits a Markovian optimal control $\alpha^\star$ satisfying
	\[
	\alpha_t^\star = - \Zc_t= -  u(t,X_t),\,  \d \P\otimes \d t\ae~\text{on } [0,T]\times\Omega,
	\]
	where $(\Yc,\Zc)$ is a pair of progressively measurable processes satisfying
	\begin{equation}
	\label{eq:pro.bsde}
		\Yc_t = Q_c\varphi(X_T) + \int_t^T\frac{1}{2}|\Zc_s|^2 ds - \int_t^T\Zc_s\d W_s, \, t\in [0,T],\P\as,
	\end{equation}
	and $u:[0,T]\times \R^\xdim \longrightarrow  \R^{m}$ a Borel-measurable function.

	If in addition $Q_c\varphi$ and $\partial_x U$ are continuously differentiable, 
	then $\alpha_t^\star = \partial_xv(t,X_t) , t \in [0,T], \P\as$, where $v$ is the classical solution of the {\rm PDE}
	\begin{equation}
	\label{eq:PDE.prop}
	\begin{cases}
	 	\partial_t v(t,x) - \frac12\partial_xv(t,x)\partial_x U(x)  +\frac12\partial_{xx}v(t,x)  -\frac12|\partial_xv(t,x)|^2 = 0,\, (t,x)\in [0,T)\times\R^\xdim\\ 
	 	 v(T,x) = Q_c\varphi(x).
	 	\end{cases}
 	 \end{equation} 
\end{proposition}
\begin{proof}


\medskip

If $\ell(x) = x\log(x) -x +1$, then the stochastic optimal control problem becomes
\begin{equation*}
	\Psi^\varphi(\mu) - \Ic_\ell(\mu|\nu_0) = \inf_{\P^\alpha\in \Ac(\mu )}\E^{\P^\alpha} \bigg[ \frac12\int_0^T  \|\alpha_t\|^2 \d t  +  Q_c \varphi(X_T) \bigg].
\end{equation*}
In this case, it is classical that the optimal control is Markovian.
See, e.g., \cite[Theorem 2.3]{lacker2017limit}.
The representation $\alpha_t^\star = -\Zc_t$
 with $(\Yc, \Zc)$ satisfying \eqref{eq:pro.bsde} follows, e.g., by \cite[Theorem A.1]{possamai2024policy}.

\medskip

Let $M$ be the continuous martingale $M_t:=\E[e^{Q_c\varphi(X_T)}|\Fc_t]$.
Let $N_t$ be the unique process such that $M_t = Q_c\varphi(X_T) - \int_t^TN_s\d W_s$.
Then, it follows by It\^o's formula applied to $\Yc_t = \log(M_t)$ that $\Zc_t = N_t/M_t$.
Moreover, by \cite[Lemma 3.3]{el1997backwardii} there are two Borel-measurable functions $e:[0,T]\times \R^\xdim\longrightarrow  \R$ and $d:[0,T]\times \R^\xdim\longrightarrow \R^d$ such that, $M_t = e(t,X_t), t\in [0,T],\P\as$, and, $N_t =  d(t,X_t),\d \P\otimes \d t\ae~\text{on } [0,T]\times\Omega$.
Thus it holds that $\Yc_t = v(t,X_t), t\in [0,T], \P\as$ and $\Zc_t =   u(t,X_t), \d \P\otimes \d t\ae~\text{on } [0,T]\times\Omega$, for two Borel-measurable functions $u$ and $v$.

 If $Q_c\varphi$ and $b$ are continuously differentiable, then it follows by \cite[Theorem 3.1]{ma2002representation} that $\alpha_t^\star =  \partial_xv(t,X_t)$ where $v$ solves the PDE \eqref{eq:PDE.prop}.
\end{proof}

\begin{proposition}
\label{prop:optimal.control.Markovian.quadratic}
	Let {\rm \Cref{assump.data}} be satisfied and put $\ell(x) = \frac{(x - 1)^2}{2}$ for $x \ge 0$ and $\ell(x)=+\infty$ for $x<0$.
	Assume that $Q_c\varphi$ and $\partial_x U$ are continuously differentiable with bounded derivatives. 
	Then for any $\varphi\in B_{b,p}(\R^\xdim)$ the problem \eqref{eq:cont.problem} admits a Markovian optimal control $\alpha^\star$ satisfying
	\begin{equation}
	\label{eq:optim.con.chisquare}
		\alpha_t^\star =  -\frac{ \partial_xv(t,X_t)}{Z_t}, \, \d  \P\otimes \d t\ae~\text{on }[0,T]\times\Omega,
	\end{equation} 
	where $v$ and $\tilde v$ are classical solutions of the respective {\rm PDEs}
	\begin{equation}
	\label{eq:PDE.prop.quadratic1}
	\begin{cases}
	 	\partial_t v(t,x) - \frac12\partial_xv(t,x)\partial_x U(x)  +\frac12\partial_{xx}v(t,x)  -\frac12| \partial_x\tilde v(t,x)|^2 = 0,\, (t,x)\in [0,T)\times\R^\xdim\\ 
	 	 v(T,x) = 0,
	 	\end{cases}
 	 \end{equation} 
 	 and 
 	 \begin{equation}
	\label{eq:PDE.prop.quadratic2}
	\begin{cases}
	 	\partial_t \tilde v(t,x) - \frac12\partial_x\tilde v(t,x)\partial_x U(x)  +\frac12\partial_{xx}\tilde v(t,x)   = 0,\, (t,x)\in [0,T)\times\R^\xdim\\ 
	 	 v(T,x) = Q_c\varphi(x).
	 	\end{cases}
 	 \end{equation} 
\end{proposition}
\begin{remark}
	The optimal control process $\alpha^\star$ is indeed Markovian because $X$ and $Z$ are both state processes.
	The crucial difference is that $\alpha^\star$ does not depend on $Z$ in the entropic case.
	We believe that this is the only case where $\alpha^\star$ does not depend on $Z$.
\end{remark}
\begin{proof}
	The Hamilton-Jacobi-Bellman (HJB) equation associated to the stochastic control problem \eqref{eq:cont.problem} is given by
	\begin{equation}
	\label{eq:hjb.quadratic}
		\begin{cases}
			\partial_t V -\frac12 \partial_x U\partial_xV + \frac12  \partial_{xx}V + \inf_{a \in \R^\xdim}\Big\{ \frac12z^2|a|^2  + \frac12(2  za\partial_{xz}V +  z^2|a|^2 \partial_{zz}V)\Big \} = 0\\ 
			V(T,x,z) = zQ_c\varphi(x).
		\end{cases}
	\end{equation}
	By the standard verification theorem in optimal stochastic control theorem, see, e.g., \cite{fleming2006controlled,touzi2013optimal}, we know that the optimal control is given by $\alpha^\star = \argmin_{a\in \R^\xdim}\Big\{ \frac12Z_t^2|a|^2  + \frac12(2  Z_ta\partial_{xz}V(t,X_t,Z_t) +  Z^2_t|a|^2 \partial_{zz}V(t,X_t,Z_t))\Big \}$ provided that $V$ is a classical solution of the HJB \eqref{eq:hjb.quadratic}.
	Thus, if $V$ is a classical solution, then we have 
	\[
	\alpha^\star_t = - \frac{ \partial_{xz}V(t,X_t, Z_t)}{Z_t(1 + \partial_{zz}V(t,X_t,Z_t))},\,  \d  \P\otimes \d t\ae~\text{on }[0,T]\times\Omega.
	\]
	To solve the HJB equation, we make the ansatz $V(t,x,z) = v(t,x) + z \tilde v(t,x)$ with $z>0$ and $(t,x)\in [0,T]\times \R^\xdim$.
	It then follows that $V$ solves the HJB equation if $v$ and $\tilde v$ satisfy \eqref{eq:PDE.prop.quadratic1} and \eqref{eq:PDE.prop.quadratic1}.
	It follows from \citet[Theorems 1.7.12 and 2.4.10]{friedman1964partial} that equation \eqref{eq:PDE.prop.quadratic2} admits a unique classical solution $\tilde v$.
	This solution has Lipschitz derivative.
	In fact, it follows by Feynman-Kac formula that $\tilde v(t,x) = \E[Q_c\varphi(X^{t,x}_T)]$ where $X^{t,x}$ solves the SDE \eqref{eq.dyn.P} with $X_t=x$.
	Thus, we have that $\partial_x\tilde v(t,x) = \E[\partial_xQ_c(X^{t,x}_T)\nabla X^{t,x}_T]$ where $\nabla X^{t,x}$ is the variational process satisfying $\d \nabla X^{t,x}_s = -\frac12\partial_x U( X^{t,x}_s)\nabla X^{t,x}_s\d s$.
	Since $\partial_x U$ is Lipschitz continuous and $\partial_xQ_c\varphi$ is bounded, it is direct to check that $\partial_xv$ is Lipschitz.
	Then applying \citet[Theorems 1.7.12 and 2.4.10]{friedman1964partial}  again, \eqref{eq:PDE.prop.quadratic2} also admits a unique classical solution.
	This shows that $V$ solves the HJB equation and therefore that an optimal control is given by $\alpha_t =  -\frac{ \partial_xv(t,X_t)}{Z_t}$.
\end{proof}
\begin{lemma}
\label{lem:phi.optimal=primal.optimal}
	Let {\rm Assumptions \ref{assump.data}} and {\rm \ref{assump.data.2}} be satisfied.
	Assume that $\varphi^\star \in B_{b,p}(\R^\xdim)$ is a dual optimizer for $V_c(\muin,\mufin)$ and $\alpha^\star$ is optimal for $\Psi^{\varphi^\star}(\muin)$.
	Then the measure $\P^{\alpha^\star}$ 
	is the primal optimizer for $V_c(\muin,\mufin)$.
	In particular, $\P^{\alpha^\star} = \Q^\star$.
\end{lemma}
\begin{proof}
Let us first assume that $(i)$ in Assumption \ref{assump.data.2} holds.	
	If $\varphi^\star \in C_{b,p}(\R^\xdim)$ is dual optimal for $V_c(\muin,\mufin)$ and $\alpha^\star$ is optimal for $\Psi^{\varphi^\star}(\muin) - \Ic_\ell(\muin|\nu_0)$, then subsequently using Corollary \ref{cor.chain} and the desintegration theorem, optimality of $\alpha^\star$, definition of $\Psi^{\varphi^\star}(\muin)$ and \Cref{prop:dual.f.div}, we have
	\begin{align*}
		\int_{\R^\xdim}\Big( \Ic_\ell(\P^{\alpha^\star}_x|\P_x)& + \E^{\P^{\alpha^\star}_x}[Q_c \varphi^\star(X_T)]\Big)\mu_0(dx)  =
		 \Ic_\ell(\P^{\alpha^\star}|\P) + \E^{\P^{\alpha^\star}}[Q_c \varphi^\star(X_T)]\\ 
		  &= \Psi^{\varphi^\star}(\muin)    
			 = -\int_{\R^m}\Phi_{\P_x}(-Q_c \varphi^\star(X_T))\muin(dx)\\ 
			& = -\int_{\R^m}\sup_{\Q\ll \P_x}\big(   \E^{\Q}[-Q_c\varphi^\star(X_T)] - \Ic_\ell(\Q|\P^x)\big)\muin(dx).
	\end{align*}
If we instead have $(ii)$ in Assumption \ref{assump.data.2}, then by superadditivity of $\Ic_\ell$ and arguing as above 
\begin{align*}
	\int_{\R^\xdim}\Big( \Ic_\ell(\P^{\alpha^\star}_x|\P_x) + \E^{\P^{\alpha^\star}_x}[Q_c \varphi^\star(X_T)]\Big)\mu_0(dx) & \le 
		 \Ic_\ell(\P^{\alpha^\star}|\P) + \E^{\P^{\alpha^\star}}[Q_c \varphi^\star(X_T)]\\ 
			& = -\int_{\R^m}\sup_{\Q\ll \P_x}\big(   \E^{\Q}[-Q_c\varphi^\star(X_T)] - \Ic_\ell(\Q|\P^x)\big)\muin(dx).
	\end{align*}
	Since $\P^{\alpha^\star}_x\ll \P_x$, this yields
	\begin{equation*}
		\int_{\R^\xdim}\Big( \Ic_\ell(\P^{\alpha^\star}_x|\P_x) + \E^{\P^{\alpha^\star}_x}[Q_c \varphi^\star(X_T)]\Big)\mu_0(dx) = -\int_{\R^m}\sup_{\Q\ll \P_x}\big(   \E^{\Q}[-Q_c\varphi^\star(X_T)] - \Ic_\ell(\Q|\P^x)\big)\muin(dx).
	\end{equation*}
This implies that $\Ic_\ell(\P^{\alpha^\star}_x|\P_x) + \E^{\P^{\alpha^\star}_x}[Q_c \varphi^\star(X_T)] = -\sup_{\Q\ll \P_x}\big(   \E^{\Q}[-Q_c\varphi^\star(X_T)] - \Ic_\ell(\Q|\P^x)\big), \muin\ae$
	Thus, it follows by uniqueness of the optimizer in \Cref{prop:dual.f.div} that $\frac{\d\P^{\hat\alpha}_x}{\d \P_x} = \partial_x \ell^*\big(-Q_c\varphi^\star(X_T) -  \phi^\star_x\big), \, \P^\star_x\as$ for $\muin\ae\, x\in \R^\xdim$ with $\phi^\star_x$ as in \eqref{eq.thm.i.aux.4}.
Arguing as in the proof of {\rm \Cref{thm:propertyMain}}, this shows that $\P^{\alpha^\star} = \Q^\star$, which concludes the proof.
\end{proof}

\subsubsection{Time reversal of diffusions and proofs of the characterizations of the marginal distribution}

Recall the time reversal mapping $\overleftarrow{\Tc}$ defined in the introduction.
Observe in particular that with respect to the canonical process $X$ we have
\[
	X_t(\overleftarrow \Tc(\omega)):=X_{T-t}(\omega)=\omega_{T-t}.
\]
Recall that a probability measure is said to be reversible if $\overleftarrow{\Q} = \Q$.
It is a classical result, see \citet[Proposition 4.5]{pavliotis2014stochastic}, that whenever the drift $b$ is Markovian and of gradient form, i.e., $b(t,x) = -\partial_x U(x(t))/2$, $(t,x)\in [0,T]\times C([0,T],\R^\xdim)$, for some continuously differentiable function $U:\R^\xdim\longrightarrow \R$, the SDE \eqref{eq.dyn.P} admits an invariant measure $\varlambda$ and taking $\nu_0= \varlambda$ it follows that $\P$, the weak solution to \eqref{eq.dyn.P}, is reversible.

\begin{proposition}
\label{prop:charactirization.final}
	Let {\rm \Cref{assump.data}} be satisfied, $\nu_0=\varlambda$, and assume that $c$ is such that $\Wc_c(\mu,\nu)=0$ if and only if $\mu = \nu$ and $b(t,x) = -\partial_x U(x(t))/2$, $(t,x)\in [0,T]\times C([0,T],\R^\xdim)$, for some continuously differentiable function $U:\R^\xdim\longrightarrow \R$.
	Let $\Phi^\varphi$ be given by \eqref{eq:defPsi}.
	Then the following hold:
	\begin{enumerate}[label=$(\roman*)$, ref=.$(\roman*)$,wide,  labelindent=0pt]
	\item The optimal transport problem with value $V_c(\mufin, \muin)$ admits the convex dual representation
	\begin{equation}
	\label{eq:V.dual.sign.invers}
		V_c(\mufin,\muin) = \sup_{\varphi\in C_{b,p}(\R^\xdim)}\big(  \Psi^\varphi(\mufin) + \langle \varphi,\muin\rangle\big).
	\end{equation}
	\item If $\Pc_c(\mufin,\muin) \neq \emptyset$, then $V_c(\mufin,\muin)$ admits a unique primal optimizer $\hat \Q$ and it holds $\hat \Q = \overleftarrow{\Q^\star}$.
	\item Assume that $V_c(\muin,\mufin)$ admits a dual optimizer $\varphi^\star\in C_{b,p}(\R^\xdim)$
	 and $V_c(\mufin,\muin)$ admits a dual optimizer
	 $\overleftarrow\varphi \in C_{b,p}(\R^\xdim)$.
	 Then there are two progressively measurable processes  $\alpha^\star$ and $\hat{\alpha}$ such that
	\begin{equation}
	\label{eq:non-markovian.charac}
		\E^{\Q^{\star}}\Big[ \alpha^\star_t  + \hat{\alpha}_{T-t} \circ \overleftarrow{\Tc} |X_t = x \Big] - \partial_x U(x)  = \nabla \log(\Q^{\star}_t(x)),\;  \d t\otimes \d\Q_t^\star\ae~\text{on } [0,T]\times\R^\xdim,
	\end{equation}
	where $\Q^\star_t$ is the probability density function of $\Q^\star\circ X_t^{-1}$, and {where the derivative is in the sense of distributions}.
	In particular, $\alpha^\star$ and $\hat\alpha$ are the optimal controls of the stochastic optimal control problem with value $\Psi^{\phi^\star}(\muin) - \Ic_\ell(\muin|\nu)$ and $\Psi^{\overleftarrow{\varphi}}(\mufin) - \Ic_\ell(\mufin|\nu_0)$, respectively.
\end{enumerate}
 \end{proposition}

\begin{proof}
	{\rm $(i)$}: The representation \eqref{eq:V.dual.sign.invers} follows from the convex duality \Cref{thm:exists.duality}{\rm (ii)}.

	\medskip

	{\rm $(ii)$}: If $\Pc_c(\mufin,\muin)\neq \emptyset$, then the existence of a unique primal optimizer $\hat\Q$ follows by \Cref{thm:exists.duality}{\rm (i)}.

	Since $\frac{\d  \Q}{\d \P} $ is clearly a measurable function of $ \overleftarrow \Tc$ we obtain by \Cref{lem:time.reversed.divergence} that
	$\Ic_\ell(\Q |\P )= \Ic_\ell(\overleftarrow \Q|\overleftarrow \P)$.
	In addition, using that $\Wc_c(\mu,\nu) = 0$ if and only if $\mu = \nu$, it is directly checked that $\Q\in \Pc_c(\mufin,\muin)$ if and only if $\overleftarrow{\Q} \in \Pc_c(\muin, \mufin)$.
	Furthermore, since the drift is a time-independent gradient, it follows, {\color{black}e.g., by \cite[Proposition 4.5.]{pavliotis2014stochastic}} that $\P$ is reversible.
	This allows us to get
	\begin{align*}
		\Ic_\ell(\overleftarrow{\Q^\star}|\P) &\ge V_c(\mufin,\muin) = \inf_{\Q\in \Pc_c(\mufin,\muin)}\Ic_\ell(\Q|\P)=\inf_{\Q\in \Pc_c(\mufin,\muin)}\Ic_\ell(\overleftarrow{\Q}|\overleftarrow{\P})\\ 
		& = \inf_{\Q\in \Pc_c(\muin,\mufin)}\Ic_\ell(\Q|\P) = \Ic_\ell(\Q^\star|\P) = \Ic_\ell(\overleftarrow{\Q^\star}|\overleftarrow{\P})  = \Ic_\ell(\overleftarrow{\Q^\star}|\P),
	\end{align*}
	showing by uniqueness that $\hat\Q = \overleftarrow{\Q^\star}$.
\medskip

	{\rm $(iii)$}
	We have already argued (just before Proposition \ref{prop:optimal.control.Markovian.entropy}) that there exists a control process $\alpha^*$ that is optimal for $\Psi^{\varphi^\star}(\muin) - \Ic_\ell(\muin|\nu_0)$. 
	And arguing the same way but replacing $Q_c\varphi$ by $\overleftarrow{\varphi} \in C_{b,p}(\R^d)$, it follows that $\Psi^{\overleftarrow{\varphi}}(\mufin) - \Ic_\ell(\mufin|\nu_0)$ admits an optimal control $\hat\alpha$.
	Using \Cref{lem:phi.optimal=primal.optimal} now allows to get that $\P^{\alpha^{\!\star}} = \Q^\star$ and $\P^{\hat{\alpha}} = \hat \Q = \overleftarrow{\Q^\star}$.
	Thus, $\P^{\hat{\alpha}} =\overleftarrow{\P^{\alpha^{\!\star}}} $.
	By Girsanov's theorem the canonical process $X$ hence satisfies
	\begin{equation}
	\label{eq:proof.reverse.1}
	\begin{cases}
		\d X_t = \alpha_t^\star -\frac12 \partial_x U( X_t)dt + \d W_t^{\alpha^\star}, \, \Q^\star\as,  \text{ with } \Q^\star\circ X_0^{-1} = \muin\\ 
		\d Z^{\alpha^\star}_t = |\alpha^\star_t|^2Z^{\alpha^\star}_tdt + \alpha^\star_tZ^{\alpha^\star}_tdW^{\alpha^\star}_t, \, \Q^\star \as 
	\end{cases}
	\end{equation}
	and
	\begin{equation}
	\label{eq:proof.reverse.2}
		\begin{cases}
			\d X_t = \hat{\alpha}_t -\frac12\partial_x U( X_t)dt + \d W_t^{\hat{\alpha}}, \,  \overleftarrow{\Q^\star}\as , \text{ with } \overleftarrow{\Q^\star}\circ X_0^{-1} = \mufin
	\\ 
	\d Z^{\hat \alpha}_t = |\hat \alpha_t|^2Z^{\hat \alpha}_tdt + \hat \alpha_tZ^{\hat \alpha}_tdW^{\hat \alpha}_t , \,  \overleftarrow{\Q^\star}\as
\end{cases}
\end{equation}
where $W^{\alpha^\star}$ and $W^{\hat{\alpha}}$ are $\Q^\star-$ and $\overleftarrow{\Q^\star}-$Brownian motions, respectively.
Thus, by \cite[Theorem 3.10]{follmer1986time}, the law $\Q^\star_t := \Q^\star\circ X_t^{-1}$ is absolutely continuous and its density $\Q^*_t(\cdot)$ satisfies
\begin{equation*}
	\E^{\Q^{\star}}\Big[ \alpha^\star_t -\frac12\partial_x  U( X_t) + \{\hat{\alpha}_{T-t} -\frac12\partial_x U(X_{T-t})\}\circ \overleftarrow{\Tc}  |X_t = x \Big] = \nabla  \log(\Q^{\star}_t(x)),\;  \d t\otimes \d\Q_t^\star\ae~\text{on } [0,T]\times\R^\xdim,
\end{equation*}
Therefore, we obtain \eqref{eq:non-markovian.charac}.
\end{proof}

\begin{proof}[Proofs of {\rm \Cref{them:charac.marginal}} and {\rm\Cref{cor:rel.ent.cha}}]
	To conclude the proof of Theorem \ref{them:charac.marginal}, observe that \eqref{eq:non-markovian.charac.statement} was derived in \Cref{prop:charactirization.final}.

\medskip

Let $\ell(x) = x\log(x) -x + 1$, then it follows by \Cref{prop:optimal.control.Markovian.entropy} that $\alpha_t^\star = f_t(X_t)$ and $\hat\alpha_t = g_t(X_t)$ for two Borel measurable functions.
In particular, $\alpha^\star$ and $\hat\alpha$ do not depend on $Z^{\alpha^*}$ or $Z^{\hat\alpha}$.
Plugging this in \Cref{eq:non-markovian.charac.statement} directly yields \Cref{cor:rel.ent.cha}{\rm (i)}.
Under the additional assumptions in \Cref{cor:rel.ent.cha}{\rm (ii)},
it follows by \Cref{prop:optimal.control.Markovian.entropy} that the optimal control $\alpha^\star$ is given by  $\alpha^\star_t = \partial_xv(t,X_t)$ where $v$ is the classical solution of the PDE \eqref{eq:PDE.prop}. 
Arguing exactly as in \Cref{prop:optimal.control.Markovian.entropy} with $Q_c\varphi$ therein replaced by $\overleftarrow{\varphi}$, we have again that $\hat\alpha_t = \partial_x\overleftarrow{v}(t,X_t)$ where $\overleftarrow{v}$ satisfifes the PDE \eqref{eq:PDE.prop} but with terminal condition $\overleftarrow{\varphi}$.
Plugging these values of $\alpha^\star$ and $\hat\alpha$ in \Cref{eq:non-markovian.charac.statement} yields the result.
\end{proof}

\begin{proof}[Proof of {\rm \Cref{cor:quad.cha}}]
Let $\ell(x) = (x -1)^2/2$ if $x\ge0$ and $+\infty$ else; then we know by \Cref{prop:optimal.control.Markovian.quadratic} that the optimal control in \eqref{eq:cont.problem} is given by \eqref{eq:optim.con.chisquare}.
Thus, the equation \ref{eq:proof.reverse.1} becomes
\begin{equation}
\label{eq:proof.reverse.111}
\begin{cases}
	\d X_t =  - \frac{  \partial_xv(t,X_t)}{Z^{\alpha^\star}_t} - \frac12 \partial_x U(X_t) dt + \d W_t^{\alpha^\star} \quad \Q^\star \text{-a.s. with } \Q^\star\circ X_0^{-1} = \muin\\ 
	\d Z^{\alpha^\star}_t = \frac{|  \partial_xv(t,X_t)|^2}{Z^{\alpha^\star}_t}dt -  \partial_xv(t, X_t)\cdot dW^{\alpha^\star}_t \quad \Q^\star \text{-a.s.}
\end{cases}
\end{equation}
 
That is, the measure $\Q^\star$ solves the martingale problem with drift $B$ and volatility $\Sigma$, where
\begin{equation}
	B_t(x,z) =
	\begin{pmatrix}
-  \frac{\partial_xv(t,x)}{z} -\frac12 \partial_x U (x)\\
\frac{|\partial_xv(t,x)|^2}{z}  
\end{pmatrix}
\quad \mbox{ and } \Sigma_t(x,z) = 
\begin{pmatrix}
 I_\xdim\\
-  \partial_xv(t,x)^\top
\end{pmatrix}
\end{equation}
and $I_\xdim$ being the identity matrix of $\R^{\xdim\times \xdim}$.
That is, the canonical process $\Yc$ satisfies
\begin{equation}
\label{eq:proof.reverse.222}
	d\Yc_t = {B}_t( \Yc_t)dt + {\Sigma}_t( \Yc_t)dW_t,\quad {\Q^\star}\mbox{-a.s.}
\end{equation}
By construction, we have $\Hc(\Q^*|\P)<\infty$, where $\P$ solves the martingale problem with drift $\Bc$ and volatility $\Sigma$, with 
\begin{equation*}
	 \Bc(x) = 
		\begin{pmatrix}
-\frac12 \partial_x  U (x)\\
0  
\end{pmatrix}.
\end{equation*}
Since the functions $\partial_x U$ and $\partial_xv$ are Lipschitz-continuous, $\P$ is unique.
Thus, by \cite[Theorem 1.3]{cattiaux2023time}, the time reversal $\overleftarrow{\P}$ of $\P$ satisfies the martingale problem with drift $\overleftarrow{\Bc}_{T-t}(x,z) = - \Bc(x)  + \nabla \cdot  \Sigma\Sigma^\top_t(x,z) +  \Sigma\Sigma^\top_t(x,z) \nabla  \log \P_t(x,z)  $  and volatility $\overleftarrow{\Sigma}_t(x,z) = \Sigma_{T-t}(x,z)$, where $\P_t = \P\circ(X_t,Z_t)^{-1}$.
	Note that,
	\[
	\nabla \cdot  \Sigma\Sigma^\top_t(x,z)=\begin{pmatrix} 0_m\\ -\Delta v(t,x) \end{pmatrix}.
	\]

Applying \citet[Theorem 1.16]{cattiaux2023time} again
it follows that the time reversal $\overleftarrow{\Q^\star}$ of $\Q^\star$ solves the martingale problem with drift $\overleftarrow{B}$ and volatility $\overleftarrow{\Sigma}$, with
\begin{equation*}
	\overleftarrow{B}_{T-t}(x,z) = -B_t(x,z) + \Bc(x) + \overleftarrow{\Bc}_{T-t}(x,z) +  \Sigma\Sigma^\top_t(x,z)\nabla  \log \rho_t 
	\quad \d t \Q^\star_t\text{-a.s.}
\end{equation*}
with $\rho_t = d\Q^\star_t/d\P_t$ and $\Q^\star_t = \Q^\star\circ(X_t,Z_t)^{-1}$, and where the derivative is in the sense of distributions.
Spelling this out, using the fact that $\log(\rho_t) = \log(\Q^\star_t) - \log(\P_t)$, we have
\begin{equation*}
		\overleftarrow{B}_{T-t}(x,z) = 
		\begin{pmatrix}
   \frac{\partial_xv(t,x)}{z}  + \frac12  \partial_x U(x) \\
 -\frac{|\partial_xv(t,x)|^2}{z} -\Delta v(t,x) 
\end{pmatrix}+  \Sigma\Sigma^\top_t(x,z)  \nabla \log\Q^\star_t(x,z)
\end{equation*}
On the other hand, now specializing Equation \ref{eq:proof.reverse.2} to the current choice of $\ell$ yields
\begin{equation*}
\begin{cases}
	\d X_t = -\frac{\partial_x\overleftarrow{v}(t,X_t)}{Z_t} -\frac 12 \partial_x U(X_t) \d t + \d W_t^{\hat{\alpha}} \quad \overleftarrow{\Q^\star} \text{-a.s. with } \overleftarrow{\Q^\star}\circ X_0^{-1} = \mufin
	\\ 
	\d Z^{\hat \alpha}_t = \frac{|\sigma^\top \partial_x\overleftarrow{v}(t,X_t)|^2}{Z^{\hat \alpha}_t}dt -  \partial_x\overleftarrow{v}(t, X_t)\cdot \d W^{\hat \alpha}_t \quad \overleftarrow{\Q^\star}.
\end{cases}
\end{equation*}
That is, the probability measure $\overleftarrow{\Q^\star}$ solves the martingale problem with drift $\widehat B$ and volatility $\widehat \Sigma$ given by
\begin{equation}
	\widehat B_t(x,z) =
	\begin{pmatrix}
-  \frac{\partial_x\overleftarrow v(t,x)}{z} -\frac 12 \partial_x U(x)  \\
\frac{|\partial_x\overleftarrow v(t,x)|^2}{z}  
\end{pmatrix}
\quad \mbox{and} \quad  \widehat \Sigma_t(x,z) = 
\begin{pmatrix}
I_\xdim  \\
 \partial_x\overleftarrow v(t,x)
\end{pmatrix}.
\end{equation}
In other words, under $\overleftarrow{\Q^\star}$ the canonical process $\Yc$ on $\Cc_\xdim\times \Cc_1$ satisfies the SDE
\begin{equation}
\label{eq:proof.reverse.222_111}
	d\Yc_t = \widehat B_t( \Yc_t)dt + \widehat \Sigma_t(\Yc_t)dW_t,\quad \overleftarrow{\Q^\star}\mbox{-a.s.}
\end{equation}
By \eqref{eq:proof.reverse.222_111} and \eqref{eq:proof.reverse.222} we thus deduce that $\widehat B = \overleftarrow{B}$ and $\widehat \Sigma = \overleftarrow{\Sigma}$.
That is,
\begin{equation*}
		\begin{pmatrix}
- \frac{\partial_x\overleftarrow v(T-t,x)}{z} -\frac12  \partial_x U(x)  \\
\frac{|\partial_x\overleftarrow v(T-t,x)|^2}{z}  
\end{pmatrix} = 
		\begin{pmatrix}
   \frac{\partial_x  v(t,x)}{z}  +\frac 12 \partial_x U(x)\\
 -\frac{|\partial_x  v(t,x)|^2}{z}  -\Delta v(t,x) 
 \end{pmatrix} +  \Sigma\Sigma^\top_t(x,z)  \nabla \log\Q^\star_t(x,z),
\end{equation*}
which is tantamount to the desired result.
\end{proof}

\begin{appendix}
\section{Appendix}
In this appendix we provide full justification or references for some steps in the proofs.
Some of the results might be well-known, we provide details because we could not find a directly citable reference.
We begin by recalling the duality properties of divergences.

\begin{lemma}
\label{prop:dual.f.div} 
Let $\ell$ satisfy {\rm \Cref{assump.data}\ref{assump.data.1}}.
For every $\bar\P\in{\rm Prob}(\Omega)$, the functions $\Ic_\ell(\cdot |\bar \P)$ and $\Phi_{\bar \P}$ are conjugate of one another. More precisely, it holds
\begin{align}\label{duality.f.div}
\Ic_\ell(\Q|\bar \P)=\sup_{\xi \in \L^0} \big( \E^{\Q}[\xi ]-\Phi_{\bar\P}(\xi) \big)\quad \mbox{for all } \Q\ll \bar\P
\end{align}
and 
	\begin{equation}
	\label{eq:dual.rep.gen.oce}
		\Phi_{\bar\P} (\xi) = \sup_{\Q\ll \bar \P}\big( \E^\Q[\xi] - \Ic_\ell(\Q|\bar \P) \big), \;\mbox{for all } \xi \mbox{ such that }  \E^{\bar \P}[\ell^\ast(\xi^+)]  < + \infty.
	\end{equation}
	Moreover, for each $\xi \in \L^0$ such that  $\E^{\bar \P}[\ell^\ast(\xi^+)]<\infty$, there is $ \Q^\star\ll \bar \P$ such that
		$\Phi_{\bar\P} (\xi) = \E^{\Q^\star}[\xi] - \Ic_\ell(\Q^\star|\bar \P)$.
	Such an optimizer is given by
	\begin{equation}\label{eq.dual.f.div}
		\dfrac{\d \Q^\star\!}{\d\bar \P} = \partial_x\ell^*(\xi -r^\star ),\, \bar \P\as,
	\end{equation}
	where  $r^\star \in \R$ is uniquely defined by $\E^{\bar \P}[\partial_x\ell^*(\xi - r^\star)] = 1$.
	This constant also satisfies $\Phi_{\bar\P} (\xi)  = \E^{\bar\P}[\ell^*(\xi - r^\star)] + r^\star$.
\end{lemma}

\begin{proof}
The duality statement is the variational representation of $f$-divergences, see \cite[Theorem 4.4]{bental2007old} or \cite[Theorem 7.27]{polyanskiy2025information}.
The existence and characterization of optimizers is \cite[Theorem 2.3]{cherny2007divergence}.
\end{proof}

The following result can be seen as (an instance of) the data processing inequality, well known in the entropic case.
	In the context of mathematical statistics, it pertains to the so-called sufficiency of a statistic $f:\Omega\longrightarrow \Omega$.
	The following result appeared first in \cite{liese2006divergences}, though in a slightly different setting due to the formulation of divergences.
	For the sake of completeness and the reader's convenience, we present a proof. 

\begin{lemma}
\label{lem:time.reversed.divergence}
	For any measurable map $f:\Omega\longrightarrow\Omega$ we have that
	\begin{align}\label{eq.datatransfer}
		\Ic_\ell(\Q\circ f^{-1}|\P\circ f^{-1})\leq \Ic_\ell(\Q|\P)
	\end{align}
	and whenever $\Ic_\ell(\Q|\P)<\infty$, the previous inequality is tight if and only if $\frac{\d \Q}{\d\P}$ is a measurable function of $f$.
\end{lemma}
\begin{proof}
	Let us first argue \eqref{eq.datatransfer}. 
	Note that when $\Q\ll\P$ does hold, \eqref{eq.datatransfer} is clear.
	Suppose now $\Q\ll\P$ and let $\hat\xi=\xi\circ f$ for $\xi\in\L^\infty$ and note that by definition \( \Phi_\P(\xi \circ f)= \Phi_{\P\circ f^{-1}}(\xi) \).
	Thus, thanks to \eqref{duality.f.div},
	\[
		\Ic_\ell(\Q|\P)\geq \E^{\Q\circ f^{-1}}[\xi ]-\Phi_{\P\circ f^{-1}}(\xi).
	\]
	Taking the supremum over $\xi\in \L^0$, we deduce \eqref{eq.datatransfer}.
\medskip

	Suppose now that 
	$\frac{\d \Q}{\d\P}=h\circ f$ for some measurable function $h$. 
	We claim that $\frac{\d \Q\circ f^{-1} }{\d\P\circ f^{-1}}=h $. 
	Indeed, for all $\xi$ we have
	\[
		\int_\Omega \xi(\omega)  \Q\circ f^{-1}(\d \omega)=\int_\Omega \xi(f(\omega) )  \Q (\d \omega)=\int_\Omega   \xi(f(\omega) )  h(f(\omega) )  \P (\d \omega)=\int_\Omega \xi(\omega) h(\omega) \P\circ f^{-1}(\d \omega)
	\]
	where the second equality follows by the Radon-Nikodym theorem.
	Consequently,
	\[
		\Ic_\ell(\Q|\P)=\E^\P\Big[ \ell\Big( \frac{\d \Q}{\d\P}\Big)\Big] =\E^\P[ \ell (  h\circ f)]=\E^{\P\circ f^{-1}}[ \ell (  h  ) ]=\E^{\P\circ f^{-1}}\Big[ \ell\Big(   \frac{\d \Q\circ f^{-1} }{\d\P\circ f^{-1}}\Big)\Big]=\Ic_\ell(\Q\circ f^{-1}| \d\P\circ f^{-1}) .
	\]

	Let us now assume that $\Ic_\ell (\Q\circ f^{-1}|\P\circ f^{-1})= \Ic_\ell (\Q |\P)<\infty$. Since $\Ic_\ell (\Q\circ f^{-1}|\P\circ f^{-1})<\infty$, we deduce that there exists $h$ such that for any $\xi$, 
\(
\int_\Omega \xi(\omega)  \Q\circ f^{-1}(\d \omega) =\int_\Omega \xi(\omega) h(\omega) \P\circ f^{-1}(\d \omega),
\)
and thus
\[
\int_\Omega \xi(f(\omega) )  \Q (\d \omega)=\int_\Omega \xi(\omega)  \Q\circ f^{-1}(\d \omega)=\int_\Omega \xi(\omega) h(\omega) \P\circ f^{-1}(\d \omega)=\int_\Omega   \xi(f(\omega) )  h(f(\omega) )  \P (\d \omega).
\] 
By the $\P\as$ uniqueness of $\dfrac{\d\Q}{\d\P}$, it follows again by Radon-Nikodym theorem that it is a measurable function of $f$.
\end{proof}


\begin{lemma}
\label{lemma.tv.ctransform}
Let $c(x,\rho)=\int_{\R^\xdim} \1{x\neq y}\rho(\d y)$. Then $\Wc_c(\mu,\mufin)={\rm TV}(\mu,\mufin)$ and $Q\varphi(x)=\varphi(x)$.
\begin{proof}
That $\Wc_c(\mu,\mufin)={\rm TV}(\mu,\mufin)$ follows directly by definition of WOT, see \eqref{eq.weaktransport}. 
	We now argue that, without loss of generality, we may assume that $Q\varphi(x)=\varphi(x)$. 
	First, note that since $c$ is bounded, we have that the set $C_{b,p}$ in the dual formulation \eqref{eq:dualprobl.intro} corresponds to $\varphi$ bounded.
	Second, by definition of $Q_c\varphi(x)$, see \eqref{eq:defPsi}, for any $k\in \R$, $Q_c(\varphi+k)=Q_c\varphi+k$ and $Q_c\varphi(x)=\min\{ 1+\inf_{y\neq x} \varphi(y) ,\varphi(x)\}$.
	Thus, thanks to \Cref{prop:dual.f.div} we see that
	\[
	\Psi^{\varphi+k}(\mu_0)-\langle \varphi+k,\mufin\rangle=\Psi^{\varphi}(\mu_0)-\langle \varphi,\mufin\rangle.
	\]
	Thus, we can always add a constant to $\varphi$, without increasing $\Wc_c$ and thus $Q_c\varphi(x)=\min\{0,\varphi(x)\}$.
\end{proof}
\end{lemma}

\begin{lemma}
\label{lemma.v.disintegration}
The functional 
\[ 
\Pc(\R^\xdim)\ni\nu\longmapsto \Phi(\nu )\coloneqq  \inf_{\P^\alpha\in \Ac(\nu)}\E^{\P^\alpha} \bigg[ \frac12  \int_0^T  \ell^{\prime\prime}(Z_t)Z_t \|\alpha_t\|^2 \d t  +  Q_c \varphi(X_T) + C(X) \bigg]
\]	
satisfies	
\[
\Phi(\nu )=\int_{\R^\xdim} \Phi(\delta_x)\nu (\d x).
\]
\begin{proof}
Note that $\Phi(\nu )$ denotes the value of a stochastic control problem.
	For control problems in the so-called weak formulation in the canonical space, the statement follows from \cite[Lemma 4.5]{hernandez2025schrodinger} or \cite[Lemma 3.5]{tan2013optimal} for mean-field and classical control problems, respectively.
	In either setting, the argument rests on a measurable selection argument.
	In the context of this manuscript, we may enlarge the canonical space so that both $\alpha$ and $Z$ are canonical processes and minimize over probability laws of $(X,Z,\alpha)$ instead of over $\Ac(\nu)$.
	That the value remains unchanged with such enlargement follows from \cite[Theorem 4.10]{karoui2015capacities2}.
\end{proof}
\end{lemma}

\end{appendix}

\bibliography{Bibliography}
\end{document}